\documentclass[times,final,11pt]{elsarticle}

\usepackage{amssymb}
\usepackage{subcaption} 
\usepackage{empheq}
\usepackage{times}
\usepackage{soul}
\usepackage{url}
\usepackage{tikz}
\usepackage{fullpage}
\usepackage{graphicx} 
\graphicspath{{Figures/}} 

\usepackage[shortlabels]{enumitem} 

\usepackage{lipsum} 

\usepackage{amsmath,amssymb,amsthm,mathtools} 

\usepackage{varioref} 

\usepackage{multirow} 

\theoremstyle{definition} 

\theoremstyle{plain} 
\newtheorem{theorem}{Theorem}
\newtheorem{corollary}{Corollary}[theorem]

\theoremstyle{remark} 
\newtheorem{remark}{Remark}

\usepackage{accents}
\newcommand*{\dt}[1]{
  \accentset{\mbox{\large\bfseries .}}{#1}}

\newcommand{\bc}{{\mathbf c}}
\newcommand{\be}{{\mathbf e}}
\newcommand{\bu}{{\mathbf u}}
\newcommand{\bv}{{\mathbf v}}

\newcommand{\bx}{{\mathbf x}}
\newcommand{\by}{{\mathbf y}}
\newcommand{\bB}{{\mathbf B}}

\newcommand{\bD}{{\mathbf D}}
\newcommand{\bF}{{\mathbf F}}
\newcommand{\bI}{{\mathbf I}}
\newcommand{\bJ}{{\mathbf J}}
\newcommand{\bQ}{{\mathbf Q}}
\newcommand{\bR}{{\mathbf R}}
\newcommand{\bU}{{\mathbf U}}
\newcommand{\bP}{{\mathbf P}}

\newcommand{\bX}{{\mathbf X}}
\newcommand{\bY}{{\mathbf Y}}

\newcommand{\bPhi}{{\boldsymbol \Phi}}
\newcommand{\bPsi}{{\boldsymbol \Psi}}

\usepackage[colorlinks=true, linkcolor=blue, citecolor=blue, urlcolor=blue]{hyperref}

\usepackage[textsize=tiny]{todonotes}

\begin{document}

\begin{frontmatter}



\title{Data-Driven Reduced-Order Models for Port-Hamiltonian Systems with Operator Inference}



\author[1]{Yuwei Geng}
\ead{ygeng@email.sc.edu}
\author[1]{Lili Ju}
\ead{ju@math.sc.edu}
\author[2]{Boris Kramer}
\ead{bmkramer@ucsd.edu}
\author[1]{{Zhu Wang}\texorpdfstring{\corref{cor1}}{}}
\ead{wangzhu@math.sc.edu}
\cortext[cor1]{Corresponding author}
\address[1]{Department of Mathematics, University of South Carolina, Columbia, SC 29208, USA}
\address[2]{Department of Mechanical and Aerospace Engineering, University of California San Diego, La Jolla, CA 92093, USA}

\begin{abstract}
Hamiltonian operator inference has been developed in [\textit{Sharma, H., Wang, Z., Kramer, B., Physica D: Nonlinear Phenomena, 431, p.133122, 2022}] to learn structure-preserving reduced-order models (ROMs) for Hamiltonian systems. The method constructs a low-dimensional model using only data and knowledge of the functional form of the Hamiltonian. The resulting ROMs preserve the intrinsic structure of the system, ensuring that the mechanical and physical properties of the system are maintained. In this work, we extend this approach to port-Hamiltonian systems, which generalize Hamiltonian systems by including energy dissipation, external input, and output. Based on snapshots of the system's state and output, together with the information about the functional form of the Hamiltonian, reduced operators are inferred through optimization and are then used to construct data-driven ROMs. To further alleviate the complexity of evaluating nonlinear terms in the ROMs, a hyper-reduction method via discrete empirical interpolation is applied. Accordingly, we derive error estimates for the ROM approximations of the state and output. Finally, we demonstrate the structure preservation, as well as the accuracy of the proposed port-Hamiltonian operator inference framework, through numerical experiments on a linear mass-spring-damper problem and a nonlinear Toda lattice problem.
\end{abstract}



\end{frontmatter}



\section{Introduction}
A port-Hamiltonian (pH) system \cite{van2014port} provides a mathematical framework for modeling and controlling a wide range of physical systems, with applications in fields such as electrical circuits, thermodynamics, chemical processes, and mechanical engineering \cite{duindam2009modeling}. Unlike a traditional Hamiltonian system, which describes the evolution of a physical system in time based on energy conservation, a pH system incorporates energy dissipation and external inputs and outputs through \textit{ports}. We consider a finite-dimensional pH system of the form
\begin{subequations}\label{eq:PH_system} 
\begin{empheq}[left=\empheqlbrace]{align}
\label{eq:PH_system1}
\dt{\bx}(t) &= (\bJ-\bR)\nabla_{\bx} H(\bx(t)) +\bB \bu(t),
\\[1ex]
\label{eq:PH_system2}
\by(t) &= \bB^\intercal \nabla_{\bx} H(\bx(t)),
\end{empheq}
\end{subequations}
with $\bx (0) =\bx^0$,  where $\bx(t) \in \mathbb{R}^n$ represents the $n$-dimensional state vector, $H(\bx)$ is the Hamiltonian function, which is continuously differentiable and represents the internal energy of the system, the matrix $\bJ = -\bJ^\intercal \in \mathbb{R}^{n\times n}$ is skew-symmetric, describing the interconnection of the system's energy storage elements;  
the matrix $\bR= \bR ^\intercal \in \mathbb{R}^{n\times n}$ is symmetric and positive semi-definite, characterizing the energy dissipation in the system,
 $\bB \in \mathbb{R}^{n\times m}$ is the port matrix, which describes the modalities through which energy is imported into or exported from the system, $\bu(t) \in \mathbb{R}^m$ represents the external input vector,
 and $\by(t)\in \mathbb{R}^{m}$ is the system's output vector. Based on the structural properties of $\bJ$ and $\bR$, one can conclude that the Hamiltonian function satisfies the dissipation inequality: for any $t^2 > t^1 \geq 0$,
\begin{equation}\label{eq:Hp}
\begin{aligned}
    H(\bx(t^2)) - H(\bx(t^1)) &=\int_{t^1}^{t^2} (\nabla_{\bx}H(\bx(t))^\intercal [(\bJ-\bR)\nabla_{\bx}H(\bx(t)+\bB\bu(t)] \, {\rm{d}}t\\
    &\leq
    \int_{t^1}^{t^2} \by(t)^\intercal \bu(t) \, {\rm{d}}t.
    \end{aligned}
\end{equation}

Given an initial condition, the high-dimensional dynamical system \eqref{eq:PH_system} can be simulated using geometric numerical integration schemes \cite{hairer2006geometric}. However, when the system needs to be solved repeatedly, those high-dimensional numerical simulations become computationally expensive. Model order reduction (MOR) can be used to design a surrogate model with fewer degrees of freedom to accelerate the simulations. 
It has been successfully applied to many engineering problems, particularly those governed by differential equations. Various techniques have been developed, including the reduced basis (RB) method \cite{quarteroni2015reduced}, proper orthogonal decomposition (POD) \cite{holmes2012turbulence},  dynamic mode decomposition (DMD) \cite{kutz2016dynamic,schmid2022dynamic}, operator inference \cite{kramer2024learning,peherstorfer2016data}, and interpolatory model reduction \cite{antoulas2005approximation,antoulas2020interpolatory}, among others. A common idea of these techniques is to extract characteristic features from either training data or the model equations and build a ROM during an offline stage, which can then be used for simulations at a low cost during an online stage. For an overview of projection-based MOR for parametric dynamical systems, we refer the reader to \cite{benner2015survey,hesthaven2016certified}.  

For systems with certain structure, such as Lagrangian, Hamiltonian and pH systems, failure to satisfy the structure at the ROM level can result in unphysical solutions and unstable behavior \cite{lall2003structure,carlberg2015preserving}. For instance, the reduced-order operators for $\bJ$ and $\bR$ in \eqref{eq:PH_system} may lose their respective skew-symmetry or symmetric semi-positive definiteness properties. To address this issue, structure-preserving MOR techniques have been introduced.  
Similar to other types of MOR, structure-preserving ROMs can be constructed either intrusively or non-intrusively. 
Intrusive methods include the POD-based Galerkin projection method for Hamiltonian systems \cite{gong2017structure,miyatake2019structure,barbulescu2021efficient}, 
POD/$\mathcal{H}_2$-based Petrov-Galerkin projection method for pH systems \cite{beattie2011structure,chaturantabut2016structure},  
proper symplectic decomposition (PSD) \cite{peng2016symplectic}, a variationally consistent approach for canonical Hamiltonian systems~\cite{gruber2025variationally}, and 
RB-ROMs for Hamiltonian systems \cite{afkham2017structure,hesthaven2021structure,pagliantini2021dynamical,hesthaven2022reduced}. 
Additionally, to efficiently evaluate nonlinear terms in ROMs, the discrete empirical interpolation method (DEIM), originally developed in \cite{barrault2004empirical,chaturantabut2010nonlinear} for interpolating the nonlinear terms over selected sparse sample points, has been extended in \cite{pagliantini2023gradient} to the structure-preserving setting for Hamiltonian ROMs. 
Unlike classical model reduction approaches that use linear subspaces, the authors in \cite{sharma2023symplectic,yildiz2024data,buchfink2023symplectic} combine quadratic or nonlinear manifold learning with symplectic auto-encoders to construct an intrusive ROM for Hamiltonian systems.
Non-intrusive structure-preserving MOR techniques, on the other hand, construct ROMs by using only snapshot data and some prior knowledge of the dynamical system. Hence, they are sometimes referred to as `glass box' approaches. 
Recent work includes learning of canonical Hamiltonian ROMs based on operator inference, where the linear part of the Hamiltonian gradient is inferred through constrained least-squares solutions~\cite{sharma2022hamiltonian}. Non-canonical Hamiltonian operator inference ROMs have been developed in~\cite{gruber2023canonical}, assuming the entire Hamiltonian function is known and inferring the reduced operator associated with the linear differential operator. Gradient-preserving operator inference (GP-OpInf) has been developed for conservative \textit{or} dissipative systems in~\cite{geng2024gradient}. A tensor-based operator inference approach has been introduced in~\cite{vijaywargiya2025tensor} to construct a ROM with the Hamiltonian structure while capturing parametric dependencies.

To date, data-driven model reduction for \textit{nonlinear} pH systems remains an open problem. On the other hand, for linear pH systems, there have been some works. In \cite{cherifi2022non}, the Loewner framework is applied to derive a state-space model using only input-output time-domain data, and the pH system is then inferred by solving an associated optimization problem. In \cite{morandin2023port}, the pH dynamic mode decomposition method is developed to infer linear pH ROMs, in which the reduced-order operators for $\bJ$ and $\bR$ are found in an iterative way under the assumption that the quadratic Hamiltonian is known. 

In this paper, we propose a pH operator inference \text{(pH-OpInf)} method to construct ROMs for linear and nonlinear pH systems, {\em using snapshot data of the state, external input, and output, along with knowledge of the functional form of the Hamiltonian}. The main contributions of our work include:
\begin{enumerate}[(i),noitemsep]
    \item We propose two optimization formulations to learn the structure-preserving reduced-order operators for constructing ROMs of pH systems. In particular, the optimization in the first formulation is decoupled into two subtasks in the second formulation to facilitate its solution.
    \item We incorporate hyper-reduction into the structure-preserving ROMs for nonlinear pH systems.
    \item We derive {\em a priori} error estimates for the reduced-order approximations of the inferred structure-preserving ROM for both state and output variables.
\end{enumerate}

The remainder of this paper is organized as follows. In Section \ref{sec: SP-G}, we review a Galerkin projection-based ROM as the state-of-the-art intrusive MOR method for reducing the pH system \eqref{eq:PH_system}, and employ DEIM to ensure efficient simulation of the nonlinear ROM. Section \ref{sec: GP-OpInf} proposes the new \text{pH-OpInf} method and suggests two optimization problems to find the reduced-order operators. In Section \ref{sec: error_esti}, we analyze the resulting structure-preserving ROMs and derive {\em a priori} error estimates for both state and output approximations.
The effectiveness of the proposed ROMs is numerically demonstrated through two examples in Section \ref{sec: Num_exp}. Finally, some concluding remarks are drawn in Section \ref{sec: conclusion}.
 
\section{Background on Projection-Based Reduced-Order Models for the Port-Hamiltonian System}\label{sec: SP-G}

The POD method \cite{berkooz1993proper} has been widely used in MOR to find a set of reduced basis vectors from snapshot data, which are then used to construct the ROM evolution equations. In practice, the basis vectors are often computed using singular value decomposition (SVD) on the array of snapshots, and the associated left singular vectors are truncated to form the reduced basis. The truncation error provides empirical guidance on choosing the number of reduced basis vectors. By approximating the state variables in the subspace spanned by the reduced basis and projecting the original system onto either the same subspace or another low-dimensional space, one can construct a Galerkin or Petrov-Galerkin projection-based ROM. This workflow is intrusive, as it requires access to the matrices and vectors of the full-order model (FOM) to build the ROM. When applied to Hamiltonian or pH systems, such ROMs may alter the structural properties of the FOM and result in inaccurate and unphysical solutions. This limitation motivates the development of structure-preserving ROMs, see examples in \cite{gugercin2012structure,chaturantabut2016structure,gong2017structure}. In the following, we focus on Galerkin projection and the design of a structure-preserving
ROM for the pH system~\eqref{eq:PH_system}.

\subsection{Intrusive Structure-Preserving ROMs}
Assume that the snapshot data for the state variables of the pH system \eqref{eq:PH_system} is given in matrix form as 
\begin{flalign}
\mathbf{X} \coloneqq \left[ \mathbf{x}(t^0), \mathbf{x}(t^1), \dots, \mathbf{x}(t^{s}) \right] \in \mathbb{R}^{n\times (s+1)},
\label{eq:x}
\end{flalign}
where $\bx(t^i)$ is the state vector at the discrete time instances $t^i$, for $i=0,1,\dots,s$. 
For simplicity, we assume the snapshots are uniformly distributed, that is, $t^i = i\Delta t$ with a constant $\Delta t$ that could be the time step size used in full-order simulations or a checkpoint at which solutions are stored. 
The POD basis matrix is denoted as $\bPhi \coloneqq [\phi_1,\phi_2,\dots,\phi_r]$, where $\phi_i$ is the $i$th left singular vector of $\mathbf{X}$ associated with the singular value $\sigma_i$, and $r$ is the dimension of the POD basis that is less than the rank $d$ of $\bX$.   
The dimension $r$ is typically selected to ensure that the POD basis captures a significant portion of the snapshot energy. In particular, given a prescribed tolerance $\tau$, $r$ is chosen such that $\frac{\sum_{i=1}^r \sigma_i^2}{\sum_{i=1}^d \sigma_i^2}\geq \tau$. 
Note that $\bPhi^\intercal\bPhi ={\bf I}_r$. 
The reduced approximation employs the ansatz $\hat{\bx}(t)=\bPhi\bx_r(t)$, where $\bx_r(t) \in \mathbb{R}^r$ is the unknown POD coefficient vector.

Replacing $\bx(t)$ with the reduced approximation $\hat{\bx}(t)$, and applying Galerkin projection to \eqref{eq:PH_system}, we obtain a low-dimensional ROM for the pH system \eqref{eq:PH_system}, denoted as \text{G-ROM}:  
\begin{subequations}\label{eq:Galerkin_ph} 
\begin{empheq}[left=\empheqlbrace]{align}
\label{eq:Galerkin_ph1}
\dt{\bx}_r(t) &= \bPhi^\intercal(\bJ - \bR)\nabla_{{\bx}} H(\bPhi{\bx_r(t)}) +\bPhi^\intercal\bB \bu(t),
\\[1ex]
\label{eq:Galerkin_ph2}
\by_r(t) &=\bB^\intercal\nabla_{\bx} H(\bPhi{\bx_r(t)}),
\end{empheq}
\end{subequations}
where $\by_r(t)\in \mathbb{R}^m$ is the reduced-order approximation of the output. 
Let us define the reduced Hamiltonian $H_r(\bx_r(t)) \coloneqq H(\bPhi \bx_r(t))$, then the rate of change of the reduced Hamiltonian during the ROM simulation is
\begin{align*}
    \frac{\rm d}{{\rm d} t}H_r(\bx_r(t)) &= [\nabla_{\bx_r} H(\bPhi \bx_r(t))]^\intercal\, \frac{{\rm d}\bx_r(t)}{{\rm d}t}\\
                        & = \left[\bPhi^\intercal \nabla_{\bx} H(\bPhi\bx_r(t))\right]^\intercal\left[ \bPhi^\intercal(\bJ - \bR)\nabla_{\bx} H(\bPhi{\bx_r(t)}) +\bPhi^\intercal\bB \bu(t)\right]\\
                        &= [\nabla_{\bx} H(\bPhi\bx_r(t))]^\intercal \bPhi \bPhi^\intercal \bJ\nabla_{\bx} H(\bPhi\bx_r(t)) - [\nabla_{\bx} H(\bPhi\bx_r(t))]^\intercal \bPhi \bPhi^\intercal \bR \nabla_{\bx} H(\bPhi\bx_r(t)) \\
                        &\quad + [\nabla_{\bx} H(\bPhi\bx_r(t))]^\intercal \bPhi \bPhi^\intercal \bB \bu(t).
\end{align*}
Since the matrices $\bPhi\bPhi^\intercal \bJ$ and $\bPhi\bPhi^\intercal \bR$ do not retain the same structural properties as $\bJ$ and $\bR$, integrating the equation from $t^1$ to $t^2$ shows that $H_r(\bx_r(t^2))-H_r(\bx_r(t^1))$ is not guaranteed to remain bounded above by the integration of $\by_r(t)^\intercal \bu(t)$ over $[t^1, t^2]$, as indicated by the dissipation inequality \eqref{eq:Hp} for the FOM. Typically, the \text{G-ROM} is not ensured to be dissipative even when $u \equiv \mathbf{0}$.

To overcome the same issue in MOR for a Hamiltonian system, the PSD method has been introduced in \cite{peng2016symplectic}, which is specifically designed for systems with symplectic structures. Since the pH system \eqref{eq:PH_system} is more general, we adopt the method suggested in \cite{gong2017structure} and seek $\widetilde{{\bJ}}_r$ and $\widetilde{{\bR}}_r$ such that 
$$\bPhi^\intercal \bJ = \widetilde{{\bJ}}_r \bPhi ^\intercal\quad \text{and} \quad \bPhi^\intercal \bR = \widetilde{{\bR}}_r \bPhi ^\intercal.$$
Following a least-squares approximation, one identifies the reduced-order operators as
$$\widetilde{\bJ}_r= \bPhi^\intercal \bJ \bPhi \quad \text{and} \quad
\widetilde{\bR}_r = \bPhi^\intercal \bR \bPhi.$$
Replacing $\bPhi^\intercal \bJ$ and $\bPhi^\intercal \bR$ in \eqref{eq:Galerkin_ph} with $\widetilde{\bJ}_r\bPhi^\intercal$ and $\widetilde{\bR}_r\bPhi^\intercal$, respectively, defining  
$\widetilde{\bB}_r = \bPhi^\intercal \bB$, and noticing that 
\begin{equation}
\label{eq:gradHr}
\nabla_{\bx_r}H_r(\bx_r(t)) = \bPhi^\intercal \nabla_{\bx}H(\bPhi \bx_r(t)),
\end{equation}
we obtain a structure-preserving Galerkin projection-based ROM (\text{SP-G-ROM}) for the pH system \eqref{eq:PH_system} as
\begin{subequations}\label{eq:PH_rom_system} 
\begin{empheq}[left=\empheqlbrace]{align}
\label{eq:PH_rom_system1}
\dt{\bx}_r(t) &= (\widetilde{\bJ}_r - \widetilde{\bR}_r)\nabla_{\bx_r} H_r(\bx_r(t)) + \widetilde{\bB}_r \bu(t),
\\[1ex]
\label{eq:PH_rom_system2}
\by_r(t) &=\widetilde{\bB}_r^\intercal\nabla_{\bx_r} H_r(\bx_r(t))
\end{empheq}
\end{subequations}
with the initial condition \begin{equation}
\bx_r(0) = \bPhi^\intercal\bx^0.
\label{eq:ic}
\end{equation}
Due to the same structure as its full-order counterpart, the time derivative of the reduced Hamiltonian is
\begin{align*}
    \frac{\rm d}{{\rm d} t}H_r(\bx_r(t)) 
                        & = [\nabla_{\bx_r} H_r(\bx_r(t))]^\intercal \, \frac{{\rm d}\bx_r(t)}{{\rm d} t}\\
                        & = \left[ \nabla_{\bx_r} H_r(\bx_r(t))\right]^\intercal\left[ (\widetilde{\bJ}_r - \widetilde{\bR}_r)\nabla_{\bx_r} H_r(\bx_r(t)) + \widetilde{\bB}_r \bu(t)\right]\\
                        &\leq \left[ \nabla_{\bx_r} H_r(\bx_r(t))\right]^\intercal \widetilde{\bB}_r \bu(t) \quad(\text{as } \widetilde{\bJ}_r=- \widetilde{\bJ}_r^\intercal, \widetilde{\bR}_r = \widetilde{\bR}_r^\intercal \text{ and } \widetilde{\bR}_r \succcurlyeq \mathbf{0}) \\
                        & = \by_r(t)^\intercal \bu(t). 
\end{align*}
Thus, integrating the equation from $t^1$ to $t^2$ yields that $H_r(\bx_r(t^2)) - H_r(\bx_r(t^1))$ is bounded by $\int_{t^1}^{t^2} \by_r(t)^\intercal \bu(t)\,\rm{d}t$ from above, therefore, the dissipation inequality \eqref{eq:Hp} holds at the reduced-order level.

Since the dimension $r$ of the ROM is typically much smaller than $n$, simulating the \text{SP-G-ROM} is generally more computationally efficient than simulating the FOM. However, when $\nabla_{\bx_r} H_r(\bx_r(t))$ is nonlinear, the overall computational complexity still depends on $n$. Therefore, an additional approximation, termed hyper-reduction, is needed to accelerate the ROM simulation.

\subsection{Hyper-Reduction of Nonlinear Hamiltonian ROMs}\label{sec:hyper}
The nonlinear Hamiltonian function can be recast as  
\begin{equation*}
H(\bx(t)) = \frac{1}{2}\bx(t)^\intercal\mathbf{Q}\bx(t) + N(\bx(t)),
\end{equation*}
where $\mathbf{Q}\in \mathbb{R}^{n\times n}$ represents the quadratic part and $N(\bx)$ captures the remaining non-quadratic terms.
After defining $\bQ_r = \bPhi^\intercal\bQ\bPhi$ and $N_r(\bx_r(t)) = N(\bPhi \bx_r(t))$, we have the gradient of the reduced Hamiltonian approximation as $$\nabla_{\bx_r} H_r(\bx_r(t)) = \bQ_r \bx_r (t)+\nabla_{\bx_r}N_r (\bx_r(t)).$$
The calculation of $\nabla_{\bx_r}N_r (\bx_r(t))$ is expensive for large $n$. To reduce the computational cost, we apply the approach developed in \cite{pagliantini2023gradient}, outlined below, that extends DEIM to such a nonlinear Hamiltonian function. 
Note that both standard DEIM \cite{chaturantabut2010nonlinear,drmac2016new} and Q-DEIM \cite{drmac2016new} can be used to select interpolation points. We choose the latter because of its enhanced stability and accuracy. Since Q-DEIM is a variant of DEIM, we will continue to use `DEIM' to refer to the hyper-reduction approach and label the associated hyper-reduced ROMs. 

First, without loss of generality, $N_r(\bx_r(t))$ can be written as
\begin{equation*}
N_r(\bx_r(t)) = \sum_{i=1}^{d} c_i h_i(\bPhi\bx_r(t)) = \mathbf{c}^\intercal \mathbf{h}(\bPhi\bx_r(t)),
\end{equation*}
where $d\in \mathbb{N}, \mathbf{c} \in \mathbb{R}^d$ is a constant vector and $\mathbf{h}(\cdot)$ is a vector-valued function. 
Correspondingly, its gradient is 
$$
\nabla _{\bx_r} N_r(\bx_r(t)) = \bPhi^\intercal J_{h}^\intercal(\bPhi \bx_r(t)) \mathbf{c},
$$
in which $J_h(\cdot) \in \mathbb{R}^{d\times n}$ is the Jacobian of $\mathbf{h}(\cdot)$.
Second, the approach requires forming the nonlinear snapshot matrix 
\begin{equation*}
    M_J = [J_h(\bPhi\bPhi^\intercal\bx(t^0))\bPhi,\dots,J_h(\bPhi\bPhi^\intercal\bx(t^{s}))\bPhi]  \in \mathbb{R}^{d\times r(s+1)},
\end{equation*}
which is used to determine the nonlinear basis $\bPsi\in \mathbb{R}^{d\times m}$ and the interpolation matrix $\bP:=[\be_{\wp_1},\dots,\be_{\wp_m}]$ using DEIM, where $\be_i$ is the $i$th unit vector of $\mathbb{R}^d$ and $\{\wp_1,\dots,\wp_m\}\subset \{1,\dots,d\}$ are the interpolation indices. 
Letting $\mathbb{P} = \bPsi(\bP^\intercal \bPsi)^{-1}\bP^\intercal$,  
the nonlinear term $N_r(\bx_r(t))$ is approximated by  
\begin{equation*}
    N_{h_r}(\bx_r(t)) \coloneqq \mathbf{c}^\intercal \mathbb{P} \mathbf{h}(\bPhi\bx_r(t)).
\end{equation*}
Finally, after defining the hyper-reduced Hamiltonian 
\begin{equation*} 
H_{h_r}(\bx_r(t)) = \frac{1}{2}\bx_r(t)^\intercal\mathbf{Q}_r\bx_r(t) + N_{h_r}(\bx_r(t))
\end{equation*} 
and using the same $\widetilde{\bJ}_r$, $\widetilde{\bR}_r$ and $\widetilde{\bB}_r$ as in \eqref{eq:PH_rom_system}, we construct another structure-preserving ROM, denoted as the \text{SP-G-DEIM} model:
\begin{subequations}\label{eq:PH_rom_DEIM_system}
\begin{empheq}[left=\empheqlbrace]{align}
\dt{\bx}_r(t) &= (\widetilde{\bJ}_r-\widetilde{\bR}_r)\nabla_{\bx_r} H_{h_r}(\bx_r(t)) +\widetilde{\bB}_r \bu(t),\\[1ex]
\label{eq:PH_rom_DEIM_system2}
\by_r(t) &= \widetilde{\bB}_r^\intercal \nabla_{\bx_r} H_{h_r}(\bx_r(t)),
\end{empheq}
\end{subequations}
where 
\begin{equation}
\label{eq:gradHhr}
\nabla_{\bx_r} H_{h_r}(\bx_r(t)) = \bQ_r \bx_r(t) + \bPhi^\intercal J_h^\intercal(\bPhi \bx_r(t))\mathbb{P}^\intercal \mathbf{c},
\end{equation}
and we have the same initial condition as in \eqref{eq:ic}. 
Calculating the rate of change in $H_{h_r}(\bx_r(t))$ reveals that the dissipation inequality is satisfied at the reduced-order level.
Since $\nabla_{\bx_r} H_{h_r}(\bx_r(t))$ is calculated 
with the help of interpolation, the ROM's online simulation is faster than the \text{SP-G-ROM}. For a detailed computational complexity analysis, see \cite{pagliantini2023gradient}.

Notice that for both \eqref{eq:PH_rom_system} 
 and \eqref{eq:PH_rom_DEIM_system}, in order to generate the reduced-order matrices $\widetilde{\bJ}_r$, $\widetilde{\bR}_r$ and $\widetilde{\bB}_r$, the full-order matrices $\bJ$, $\bR$ and $\bB$ are required, hence they are intrusive approaches. However, in many real-world applications, such FOM information may not be readily available. In the next section, we introduce a nonintrusive pH-OpInf approach for constructing data-driven, structure-preserving ROMs without requiring access to the $\bJ$, $\bR$, and $\bB$ matrices of the FOM.

\section{Data-driven Reduced-order Models with the Port-Hamiltonian Operator Inference}\label{sec: GP-OpInf}
To learn structure-preserving ROMs, \text{H-OpInf} in 
\cite{sharma2022hamiltonian,gruber2023canonical} and \text{GP-OpInf} in \cite{geng2024gradient} have been introduced for both canonical and noncanonical Hamiltonian systems, and more general gradient systems, respectively. Building on these approaches, we derive a structure-preserving ROM  for the pH system \eqref{eq:PH_system}, referred to as \text{pH-OpInf-ROM}, written in the form
\begin{subequations}\label{eq:PH_rom_inf_system}
\begin{empheq}[left=\empheqlbrace]{align}
\label{eq:PH_rom_inf_system1}
\dt{\bx}_r(t) &= (\bJ_r - \bR_r)\nabla_{\bx_r} H_r(\bx_r(t)) + \bB_r \bu(t),
\\[1ex]
\label{eq:PH_rom_inf_system2}
\by_r(t) &=\bB_r^\intercal\nabla_{\bx_r} H_r(\bx_r(t)).
\end{empheq}
\end{subequations}
While the model form is equivalent to \eqref{eq:PH_rom_system}, the notation is distinct to emphasize that, numerically, the projected matrices in \eqref{eq:PH_rom_system} and the learned matrices in \eqref{eq:PH_rom_inf_system} are generally different. Next, we detail our approach to inferring $\bJ_r$, $\bR_r$ and $\bB_r$, 
assuming that the functional form of the Hamiltonian is provided.

\subsection{Data Reduction}
Given snapshot data $\bX$ as in \eqref{eq:x}, as well as the input and output data
\begin{flalign}\label{eq:UYdata}
\mathbf{U} \coloneqq \left[ \mathbf{u}(t^0), \mathbf{u}(t^1), \dots, \mathbf{u}(t^{s}) \right] 
\qquad \text{and} \qquad
\mathbf{Y} \coloneqq \left[ \mathbf{y}(t^0), \mathbf{y}(t^1), \dots, \mathbf{y}(t^{s}) \right],
\end{flalign}
we generate time derivative data by applying a finite difference operator, denoted as $\mathcal{D}_t[\cdot]$, on the state vectors: 
\begin{flalign*}
\dt{\bX}  &\coloneqq \left[ \mathcal{D}_t[{\bx}(t^0)], \mathcal{D}_t[{\bx}(t^1)], \dots, \mathcal{D}_t[{\bx}(t^s)] \right].
\end{flalign*}
For instance, 
$\mathcal{D}_t[\cdot]$ can be chosen to be the second-order finite difference operator, satisfying  
\begin{flalign*}
\mathcal{D}_t[{\bx}(t^j)] = 
\left\{
\begin{array}{ll}
(-3\bx(t^0)+4\bx(t^1)-\bx(t^2))/(2\Delta t), & j=0, \\
(\bx(t^{j+1})-\bx(t^{j-1}))/(2\Delta t), & j=1, \dots, s-1,\\
(\bx(t^{s-2})-4\bx(t^{s-1})+3\bx(t^s))/(2\Delta t), & j=s. 
\end{array}
\right.
\label{eq:dt}
\end{flalign*}
We next compute a snapshot matrix for the gradient of the Hamiltonian function 
\begin{flalign*}
\mathbf{F}  &\coloneqq \left[ \nabla_\mathbf{x} H(\mathbf{x}(t^0)), \nabla_\mathbf{x} H(\mathbf{x}(t^1)), \dots, \nabla_{\mathbf{x}} H(\mathbf{x}(t^s)) \right].
\end{flalign*}
After finding the POD basis $\bPhi$ from $\bX$, we project the high-dimensional matrices onto the POD subspace and obtain the projected data 
\begin{equation}\label{eq:XrFrdata}
    \bX_r = \bPhi ^\intercal \bX,\quad 
\dt{\bX}_r = \bPhi^\intercal \dt{\bX}, \quad \text{and} \quad 
\bF_r = \bPhi^\intercal \bF.
\end{equation}
This projection is inspired by the state equation \eqref{eq:PH_rom_system1} of the \text{SP-G-ROM}. Specially, $\bX_r$ and $\dt{\bX}_r$ represent the POD coefficient vectors to approximate the snapshots and associated time derivatives in the POD subspace, and  
$\bF_r$ provides an approximation of the gradient of the Hamiltonian function at the reduced-order level, following the definition of $\nabla_{\bx_r}H_r(\bx_r(t))$ in \eqref{eq:gradHr}.

\subsection{Two Optimization Problems for \text{pH-OpInf}}
Based on these data generated from \eqref{eq:UYdata} and \eqref{eq:XrFrdata}, one can infer $\bJ_r$, $\bR_r$ and $\bB_r$ for the \text{pH-OpInf-ROM} in \eqref{eq:PH_rom_inf_system} through optimization. 
An iterative algorithm proposed in \cite{morandin2023port} recasts the entire coefficient matrix $\left[\begin{smallmatrix} \bJ_r-\bR_r & \bB_r \\ -\bB_r^\intercal & \mathbf{0} \end{smallmatrix} \right]$
into $\boldmath \widetilde{J}-\widetilde{R}$, and then learns the skew-symmetric matrix $\boldmath \widetilde{J}$ and the symmetric and positive semi-definite matrix $\boldmath \widetilde{R}$ iteratively.
However, using the learned $\boldmath \widetilde{J}$ and $\boldmath \widetilde{R}$ cannot ensure that the (2,2) block in the coefficient matrix is zero with machine precision. Moreover, the iterative algorithm requires a delicate initial guess to ensure convergence.
In this work, we identify the reduced operators with appropriate structural properties through a subsystem-based optimization rather than a monolithic approach.

Next, we propose and analyze two approaches.  
In both, we first determine $\bD_r$, then post process it to obtain $\bJ_r \coloneqq \frac{1}{2}(\bD_r - \bD_r^\intercal)$ and $\bR_r \coloneqq -\frac{1}{2}(\bD_r+\bD_r^\intercal)$ so that we again have $\bD_r =  \bJ_r - \bR_r$.

\paragraph{\text{pH-OpInf-W}} 
We find $\bD_r$ and $\bB_r$ simultaneously by minimizing an objective function that incorporates the residuals of the state equation \eqref{eq:PH_rom_inf_system1} and the output equation \eqref{eq:PH_rom_inf_system2} at the reduced level. The residuals are balanced using a weighting parameter $\lambda_w>0$ and we get 
\begin{flalign}
\min_{\bD_r\in \mathbb{R}^{r\times r} ,\bB_r\in \mathbb{R}^{r\times m} } \frac{1}{2} \|\dt{\bX}_r - \bD_r \bF_r-\bB_r \bU\|_F^2 +  \frac{\lambda_{w}}{2} \|{\bY}^\intercal -
\bF_r^\intercal\bB_r \|_F^2, \quad \text{ s.t. }\quad \frac{1}{2}(\bD_r+\bD_r^\intercal) \preccurlyeq \boldsymbol{0}.
\label{eq:opt}
\end{flalign}
This optimization has a semi-definite constraint and can be solved by an interior point method \cite{helmberg1996interior}. 
Operator inference implementations often use regularization terms to solve unconstrained optimization problems. However, we did not observe any additional numerical improvements for our constrained optimization in the test cases investigated in Section~\ref{sec: Num_exp}, so we did not implement regularization.

\paragraph{\text{pH-OpInf-R}}  
Since both $\bY$ and $\bF_r$ are given, the above optimization can be split into two subtasks: finding $\bB_r$ by minimizing $\|{\bY}^\intercal -\bF_r^\intercal\bB_r \|_F^2$ in the first step, and then determining $\bD_r$ by minimizing $\|\dt{\bX}_r - \bD_r \bF_r-\bB_r \bU\|_F^2$ in the second step. This leads to our second approach.
First, we find $\bB_r$ from  
\begin{flalign}
\min_{\bB_r\in \mathbb{R}^{r\times m}} \frac{1}{2} \|{\bY}^\intercal -\bF_r^\intercal\bB_r \|_F^2 + \frac{\lambda_R}{2} \|\bB_r\|_F^2,
\label{eq:opt1}
\end{flalign}
where $\lambda_R>0$ is a regularization parameter; 
second, using the obtained $\bB_r$, we optimize $\bD_r$ from
\begin{flalign}
\min_{\bD_r\in \mathbb{R}^{r\times r}} \frac{1}{2} \|\dt{\bX}_r - \bD_r \bF_r - \bB_r \bU \|_F^2, \quad \text{ s.t. }\quad \frac{1}{2}(\bD_r+\bD_r^\intercal) \preccurlyeq \mathbf{0}. 
\label{eq:opt2}
\end{flalign}
In the first step, \eqref{eq:opt1} is unconstrained, and hence can be easily solved using the least-squares method, and in the second step, \eqref{eq:opt2} is again a semi-definite constrained optimization, which can be solved using an interior point method.
\subsection{Data-driven \text{pH-OpInf-ROMs}}
In both optimization problems, since $\frac{1}{2}(\bD_r+\bD_r^\intercal) \preccurlyeq \mathbf{0}$, the inferred matrix $\bR_r$ must be positive  semi-definite; and $\bR_r$ and $\bJ_r$ are symmetric and skew symmetric, respectively, due to their definitions. After substituting them into the ROM \eqref{eq:PH_rom_inf_system}, we obtain the \text{pH-OpInf-ROMs} that are structure-preserving.   
Following the terminology of the optimization formulations, we refer to them as \text{pH-OpInf-W} and \text{pH-OpInf-R}. 

Correspondingly, after incorporating the hyper-reduction approach discussed in Section \ref{sec:hyper}, we obtain the \text{pH-OpInf-DEIM} model:
\begin{subequations}\label{eq:PH_rom_inf_DEIM_system}
\begin{empheq}[left=\empheqlbrace]{align}
\dt{\bx}_r(t) &= (\bJ_r-\bR_r)\nabla_{\bx_r} H_{h_r}(\bx_r(t)) +\bB_r \bu(t), \label{eq:PH_rom_inf_DEIM_system1}
\\[1ex]
\by_r(t) &= \bB_r^\intercal \nabla_{\bx_r} H_{h_r}(\bx_r(t)).
\label{eq:PH_rom_inf_DEIM_system2}
\end{empheq}
\end{subequations}
In all the ROMs, the initial condition \eqref{eq:ic} is employed.

\section{Error Estimates for Port-Hamiltonian Operator Inference ROMs} \label{sec: error_esti}
The numerical errors of projection-based structure-preserving ROMs have been analyzed for Hamiltonian systems in \cite{chaturantabut2012state,gong2017structure,pagliantini2023gradient} and for pH systems in \cite{chaturantabut2016structure}. 
In \cite{geng2024gradient}, error estimates for learned Hamiltonian ROMs are provided. 
In this work, we continue in a similar direction and estimate the {\em{a priori}} error of the \text{pH-OpInf-ROM} approximation relative to the FOM for both the state variable and the output.
To this end, we first define 
the Lipschitz constant and the logarithmic Lipschitz constant of a given mapping $f: \mathbb{R}^\ell\rightarrow \mathbb{R}^\ell$ as
\begin{equation*}
    \mathcal{C}_{\mathrm{Lip}}[f]:= \sup_{\bu\neq \bv}\frac{\|f(\bu)-f(\bv)\|}{\|\bu-\bv\|} 
    \quad \text{and} \quad 
    \mathcal{C}_{\mathrm{log-Lip}}[f]:= \sup_{\bu\neq \bv}\frac{\left<\bu-\bv, f(\bu)-f(\bv)\right>}{\|\bu-\bv\|^2},
\end{equation*}
where $\left<\cdot, \cdot\right>: \mathbb{R}^\ell\times \mathbb{R}^\ell \rightarrow \mathbb{R}$ for any positive integer $\ell$ denotes the Euclidean inner product. 
The logarithmic norm \cite{dahlquist1958stability} is defined as 
$$
\mu(\mathbf{A}) := \sup_{\bx\neq0}\frac{\Re\Big(\langle\bx, \mathbf{A}\bx\rangle\Big)}{\langle\bx, \bx\rangle},
$$
where $\Re (\lambda)$ gives the real part of a complex number $\lambda$. For the state approximation error, we have the following result. 

\begin{theorem}
Let $\bx(t) \in \mathbb{R}^n$ be the state variable of the FOM \eqref{eq:PH_system} on the time interval $[0, T]$ and ${\bx_{r}}(t) \in \mathbb{R}^r$ be the reduced-order state variable of the \text{pH-OpInf-DEIM} system \eqref{eq:PH_rom_inf_DEIM_system} on the same interval.  
Let $\bD_r = \bJ_r-\bR_r$ in \eqref{eq:PH_rom_inf_DEIM_system} and 
suppose that $\nabla_{\bx} H(\bx)$ and $J_h(\bx)$ are both Lipschitz continuous, then the ROM state approximation error satisfies
\begin{equation}
\begin{array}{ll}
\int_0^T \|\bx-\bPhi{\bx_r}\|^2\, { \rm d} t 
&\leq 
 C(T) \ \bigg( 
\underbrace{\int_0^T \|\bx-\bPhi\bPhi^\intercal \bx\|^2\, { \rm d}t}_{\text{ \em projection error}} 
+ 
\underbrace{\int_0^T \|\dt{\bx}-\mathcal{D}_t[\bx]\|^2\, { \rm d}t}_{\text{\em data error}} \\
&\quad + 
\underbrace{\int_0^T \|\bPhi^\intercal \mathcal{D}_t[\bx]  - \bD_r \bPhi^\intercal \nabla_{\bx} H(\bx) - \bB_r \bu(t)\|^2\, { \rm d}t}_{\text{\em optimization error}}
+
\underbrace{\int_0^T \|(\mathbf{I}-\mathbb{P})J_h(\bPhi \bPhi^\intercal \bx)\bPhi \|^2\, { \rm d}t}_{\text{\em hyper-reduction error}} 
\bigg), \label{eq:error_x_bound}
\end{array}
\end{equation}
where the constant $C(T)= \max\{1+C_2^2T \alpha(T), C_3^2T \alpha(T),T \alpha(T)\} $ with 
$$C_1= \mu\left(\bPhi \bD_r\bPhi^\intercal \bQ\right)\,+\|\bD_r\bPhi^\intercal\|\,
    \mathcal{C}_{\mathrm{Lip}}[\bJ_h]\,
    \|(\bP^\intercal \bPsi)^{-1}\|\,
    \|\bc\|,\quad C_2= \|\bD_r \bPhi^\intercal\|\, \mathcal{C}_{\mathrm{Lip}}[\nabla_{\bx}H],\quad C_3=\|\bD_r\|\ \|\bc\|,$$
    {\rm and} $\alpha(T)= 4\int_0^T e^{2C_1 (T-\tau)}\,{ \rm d} \tau.$
\label{thm:1}
\end{theorem}
Four terms appear in the error bound: the first term $\int_0^T \|\bx-\bPhi\bPhi^\intercal \bx\|^2 \,{ \rm d} t$ measures the \textit{projection error} caused by projecting $\bx(t)$ onto the subspace spanned by the POD basis $\bPhi$, which has been analyzed thoroughly in \cite{kunisch2001galerkin,singler2014new}; the second one $\int_0^T \|\dt{\bx}-\mathcal{D}_t[\bx]\|^2\, { \rm d} t$ is the \textit{data error} caused by generating the time derivative snapshots using a finite difference scheme; the third one $\int_0^T\|\bPhi^\intercal \mathcal{D}_t[\bx]  - \bD_r \bPhi^\intercal \nabla_{\bx} H(\bx)- \bB_r \bu(t)\|^2\, { \rm d} t$ represents the \textit{optimization error}, which comprises the model error due to fitting the reduced-order operators $\bD_r$ and $\bB_r$ using the projected FOM data, and the numerical error arising from solving the optimization problem by numerical algorithms; and the forth one ${\int_0^T \|(\mathbf{I}-\mathbb{P})J_h(\bPhi \bPhi^\intercal \bx)\bPhi \|^2\, { \rm d}t}$ is the \textit{hyper-reduction error} caused by the DEIM approximation of the nonlinear Jacobian, which is used to accelerate the evaluation of the gradient of reduced Hamiltonian \eqref{eq:gradHhr}.
Next, we present the proof. 
\begin{proof}
Consider the \text{pH-OpInf-DEIM} model \eqref{eq:PH_rom_inf_DEIM_system} of a fixed dimension $r$, and define its approximate state error by 
\begin{equation*}
\be_{\bx}\coloneqq \bx-\bPhi {{\bx_r}},
\label{eq:e0}
\end{equation*}
which can be decomposed as $\be_{\bx} =\boldsymbol{\rho} +\boldsymbol{\theta}$ with $$\boldsymbol{\rho}\coloneqq \bx-\bPhi\bPhi^\intercal \bx,\quad \boldsymbol{\theta} \coloneqq \bPhi\bPhi^\intercal \bx-\bPhi \bx_r.$$
After defining 
$$\boldsymbol{\zeta} \coloneqq \dt{\bx} - \mathcal{D}_t[{\bx}],\quad
\boldsymbol{\eta} \coloneqq \bPhi^\intercal \mathcal{D}_t[{\bx}]  - \bD_r \bPhi^\intercal \nabla_{\bx} H(\bx) - \bB_r \bu(t), \quad\boldsymbol{\xi} \coloneqq  (\mathbf{I}-\mathbb{P})J_h(\bPhi \bPhi^\intercal \bx)\bPhi$$ 
and using equation~\eqref{eq:PH_rom_inf_DEIM_system1}, 
we have the time derivative of $\boldsymbol{\theta}$
\begin{equation}\label{eq:d_theta00_DEIM}
    \begin{aligned}
    \dt{\boldsymbol{\theta}} 
         &= \bPhi \bPhi^\intercal \dt{\bx} - \bPhi \dt{\bx}_r \\
		  &= \left(\bPhi \bPhi^\intercal\dt{\bx}  - \bPhi \bPhi^\intercal \mathcal{D}_t[\bx]\right)  \\
		  &\quad +\left( \bPhi \bPhi^\intercal \mathcal{D}_t[\bx]  - \bPhi \bD_r \bPhi^\intercal \nabla_{\bx} H(\bx)  - \bPhi \bB_r \bu(t) \right) \\
		  &\quad +\left( \bPhi \bD_r \bPhi^\intercal \nabla_{\bx} H(\bx) - \bPhi \bD_r \bPhi^\intercal \nabla_{\bx} H(\bPhi \bPhi^\intercal \bx) \right) \\
		  &\quad +\left( \bPhi \bD_r \bPhi^\intercal \nabla_{\bx} H(\bPhi \bPhi^\intercal \bx) -  \bPhi \bD_r\, [\bPhi^\intercal \bQ \bPhi\bPhi^\intercal\bx + \bPhi^\intercal J_h^\intercal(\bPhi \bPhi^\intercal \bx) \mathbb{P}^\intercal \mathbf{c}] \right) \\
            &\quad + \left( \bPhi \bD_r\,[\bPhi^\intercal \bQ \bPhi\bPhi^\intercal\bx+ \bPhi^\intercal J_h^\intercal(\bPhi \bPhi^\intercal \bx) \mathbb{P}^\intercal \mathbf{c}]- \bPhi \bD_r\,[\bPhi^\intercal \bQ \bPhi\bx_r+\bPhi^\intercal J_h^\intercal(\bPhi \bx_r)\mathbb{P}^\intercal \mathbf{c}] \right). 
    \end{aligned}        		
\end{equation}
Note that
\begin{equation}
\frac{{ \rm d}}{{ \rm d} t}\|\boldsymbol{\theta}\| 
= \frac{1}{2\|\boldsymbol{\theta}\|} \frac{{ \rm d}}{{ \rm d}t}\|\boldsymbol{\theta}\|^2
= \frac{1}{\|\boldsymbol{\theta}\|} \left<\boldsymbol{\theta},\,\dt{\boldsymbol{\theta}}\right>.
\label{eq:d_theta}
\end{equation}
By taking the inner product of equation~\eqref{eq:d_theta00_DEIM} with $\boldsymbol{\theta}$, the first four terms on the right-hand side are
\begin{align}
\left<\boldsymbol{\theta}, \bPhi \bPhi^\intercal\dt{\bx}  - \bPhi \bPhi^\intercal \mathcal{D}_t[{\bx}] \right> 
= \left<\boldsymbol{\theta}, \bPhi \bPhi^\intercal (\dt{\bx} - \mathcal{D}_t[{\bx}])\right>
 & \leq \|\boldsymbol{\theta}\|\, \|\boldsymbol{\zeta}\|,
 \label{eq:err_projH0}
\\
\left<\boldsymbol{\theta}, \bPhi \bPhi^\intercal \mathcal{D}_t[{\bx}] - \bPhi \bD_r \bPhi^\intercal \nabla_{\bx} H(\bx)  - \bPhi \bB_r \bu(t) \right> 
= \left<\boldsymbol{\theta}, \bPhi \boldsymbol{\eta}\right>
 &\leq \|\boldsymbol{\theta}\|\, \|\boldsymbol{\eta}\|,
 \label{eq:err_projH}
\\
\left<\boldsymbol{\theta}, \bPhi \bD_r \bPhi^\intercal\, \left(\nabla_{\bx} H(\bx) - \nabla_{\bx} H(\bPhi \bPhi^\intercal\bx)\right) \right>
&\leq \|\boldsymbol{\theta}\|\, \|\bD_r \bPhi^\intercal\|\, \mathcal{C}_{\mathrm{Lip}}[\nabla_{\bx}H]\, \|\boldsymbol{\rho}\|,
 \label{eq:err_proju}
\\
\left<\boldsymbol{\theta}, \bPhi \bD_r \bPhi^\intercal \nabla_{\bx} H(\bPhi \bPhi^\intercal \bx) -  \bPhi \bD_r\, [\bPhi^\intercal \bQ \bPhi\bPhi^\intercal\bx + \bPhi^\intercal J_h^\intercal(\bPhi \bPhi^\intercal \bx) \mathbb{P}^\intercal \mathbf{c}] \right>
&\leq \|\bD_r\|\, \|\bc\|\,\|\boldsymbol{\xi} \|\,\|\boldsymbol{\theta}\|.
 \label{eq:err_DEIM1}
\end{align}
For the last term, following \cite[Thm. 3.1]{chaturantabut2012state} and \cite[Thm. 3.3]{pagliantini2023gradient}, its linear part can be bounded by
\begin{equation}
    \left<\boldsymbol{\theta}, \bPhi \bD_r \bPhi^\intercal\,\bQ(\bPhi\bPhi^\intercal \bx-\bPhi \bx_r)\right>
\leq \mu\left(\bPhi \bD_r\bPhi^\intercal \bQ\right)\,\|\boldsymbol{\theta}\|^2 \label{eq:err_DEIM2}
\end{equation}
and the nonlinear part is bounded by
\begin{equation}
    \left<\boldsymbol{\theta}, \bPhi \bD_r \bPhi^\intercal\,\left(J_h^\intercal(\bPhi \bPhi^\intercal \bx) -J_h^\intercal(\bPhi \bx_r) \right) \mathbb{P}^\intercal\bc\right> 
    \leq \|\bD_r\bPhi^\intercal\|\,
    \mathcal{C}_{\mathrm{Lip}}[\bJ_h]\,
    \|(\bP^\intercal \bPsi)^{-1}\|\,
    \|\bc\|\,\|\boldsymbol{\theta}\|^2.  \label{eq:err_DEIM3}
\end{equation}

Combining \eqref{eq:d_theta} with \eqref{eq:d_theta00_DEIM} and \eqref{eq:err_projH0}-\eqref{eq:err_DEIM3}, we have
\begin{equation*}
\frac{{ \rm d}}{{ \rm d} t}\|\boldsymbol{\theta}\| \leq C_1 \|\boldsymbol{\theta}\| + C_2 \|\boldsymbol{\rho}\| +C_3 \|\boldsymbol{\xi}\|+ \|\boldsymbol{\zeta}\| + \|\boldsymbol{\eta}\|.
\label{eq:d_theta0}
\end{equation*}
By the classical differential version of Gronwall lemma over the interval $[0, t]$, for any $t\in [0, T]$, we get
\begin{equation*}
\|\boldsymbol{\theta}(t)\| \leq \int_0^t e^{C_1 (t-\tau)} \left(C_2 \|\boldsymbol{\rho}\|+C_3 \|\boldsymbol{\xi}\| + \|\boldsymbol{\zeta}\| + \|\boldsymbol{\eta}\|\right)\, { \rm d} \tau,
\label{eq:d_theta1}
\end{equation*}
in which the fact that $\boldsymbol{\theta}(0)=\mathbf{0}$ is used, ensured by the initial condition \eqref{eq:ic}.
Applying the Cauchy-Schwarz inequality to the RHS and squaring both sides, we have 
\begin{align*}
\|\boldsymbol{\theta}(t)\|^2
& \leq \int_0^t e^{2C_1 (t-\tau)} { \rm d} \tau \int_0^t \left(C_2 \|\boldsymbol{\rho}\| +C_3 \|\boldsymbol{\xi}\|+ \|\boldsymbol{\zeta}\| + \|\boldsymbol{\eta}\|\right)^2\, { \rm d} \tau ,\\
& \leq \alpha(T)\left(C_2^2 \int_0^T \|\boldsymbol{\rho}\|^2 { \rm d} t+C_3^2 \int_0^T \|\boldsymbol{\xi}\|^2 { \rm d} t+ \int_0^T \left(\|\boldsymbol{\zeta}\|^2 + \|\boldsymbol{\eta}\|^2\right)\, { \rm d} t\right),
\label{eq:d_theta2}
\end{align*}
where $\alpha(T)= 4\int_0^T e^{2C_1 (t-\tau)}\,  { \rm d} \tau$.
Hence,
\begin{equation*}
\int_0^T\|\boldsymbol{\theta}(t)\|^2\, { \rm d} t \leq T \alpha(T)\left(C_2^2 \int_0^T \|\boldsymbol{\rho}\|^2 { \rm d} t +C_3^2 \int_0^T \|\boldsymbol{\xi}\|^2 { \rm d} t+ \int_0^T \left(\|\boldsymbol{\zeta}\|^2 + \|\boldsymbol{\eta}\|^2\right)\, { \rm d} t\right).
\label{eq:d_theta3}
\end{equation*}
This, together with the orthogonality of $\boldsymbol{\rho}$ and $\boldsymbol{\theta}$, yields
$$
\int_0^T \|\be_x(t)\|^2\, { \rm d} t \leq \left(1+ C_2^2 T \alpha(T)\right) \int_0^T \|\boldsymbol{\rho}\|^2 { \rm d} t+C_3^2T \alpha(T) \int_0^T \|\boldsymbol{\xi}\|^2{ \rm d} t+T \alpha(T) \int_0^T \left(\|\boldsymbol{\zeta}\|^2 + \|\boldsymbol{\eta}\|^2\right)\, { \rm d} t,
$$
which proves the theorem. 
\end{proof}

Consequently, we can show the following estimate for the output approximation error.
\begin{theorem}
Let $\by(t) \in \mathbb{R}^m$ be the output of the FOM \eqref{eq:PH_system} on the time interval $[0, T]$ and ${\by_r}(t) \in \mathbb{R}^m$ be the output of the \text{pH-OpInf-DEIM} system \eqref{eq:PH_rom_inf_DEIM_system} on the same interval.
Assume that $\nabla_{\bx} H(\bx)$ and $\bJ_h(\bx)$ are both Lipschitz continuous, then the output approximation error satisfies
\begin{equation}
\int_0^T \|\by-{\by_r}\|^2\, { \rm d} t 
\leq  
C \ \bigg( 
\underbrace{\int_0^T \|\bx-\bPhi \bx_r\|^2\, { \rm d}t}_{\text{ \em state approximation error}} 
+
\underbrace{\int_0^T \|{\by}-\bB_r^\intercal \bPhi^\intercal \nabla_{\bx} H(\bx)\|^2\, { \rm d}t}_{\text{\em optimization error}}
+
\underbrace{\int_0^T \|(\mathbf{I}-\mathbb{P})J_h(\bPhi \bPhi^\intercal \bx)\bPhi \|^2\, { \rm d}t}_{\text{\em hyper-reduction error}} \bigg),
\label{eq:error_y_bound}
\end{equation}
where $C =4\max \{1, C_4^2, C_5^2, C_6^2 \}$ with the constants 
$$C_4 = \|\bB_r^\intercal \bPhi^\intercal\|\mathcal{C}_{\mathrm{Lip}}[\nabla_{\bx}H],\quad C_5=\|\bB_r\|\, \|\bc\|,\quad {\rm and }\quad  C_6=\| \bB_r^\intercal\bPhi^\intercal\|\, \left(\|\bQ\|\,+\mathcal{C}_{\mathrm{Lip}}[J_h] \|(\bP^\intercal \bPsi)^{-1}\|\,
    \|\bc\|\,\right).$$
\label{thm:2}
\end{theorem}
Three terms appear in this error bound: 
the first term $\int_0^T \|\bx-\bPhi \bx_r\|^2\, { \rm d}t$ represents the \textit{state approximation error}, which is estimated in Theorem~\ref{thm:1}; 
the second one $\int_0^T  \|{\by}-\bB_r^\intercal \bPhi^\intercal \nabla_{\bx} H(\bx)\|^2\, { \rm d}t$ is the \textit{optimization error} caused by inferring $\bB_r$ from the optimization problems; 
and the third one $\int_0^T \|(\mathbf{I}-\mathbb{P})J_h(\bPhi \bPhi^\intercal \bx)\bPhi \|^2\, { \rm d}t$ is the same \textit{hyper-reduction error} due to the DEIM interpolation of the nonlinear Jacobian of the reduced Hamiltonian.
Next, we present the proof.
\begin{proof}
Consider the \text{pH-OpInf-DEIM} model \eqref{eq:PH_rom_inf_DEIM_system} of a fixed dimension $r$, and define its output approximation error by 
\begin{equation*}
\be_{\by} \coloneqq \by-{\by_r}.
\label{eq:ey0}
\end{equation*}
By equation~\eqref{eq:PH_rom_inf_DEIM_system2}, it can be equivalently rewritten as 
\begin{equation}
\begin{aligned}
{\be_{\by}} &= \by- \bB_r^\intercal \nabla_{\bx_r} H_{h_r}(\bx_r) \\
        &= \by  - {\bB_r}^\intercal {\bPhi}^\intercal \nabla_{\bx} H(\bx) \\
        &+ \bB_r^\intercal \bPhi^\intercal \nabla_{\bx} H(\bx) - \bB_r^\intercal\bPhi^\intercal  \nabla_{{\bx}} H(\bPhi\bPhi^\intercal{\bx}) \\
        &+\bB_r^\intercal[\bPhi^\intercal \bQ \bPhi\bPhi^\intercal{\bx} + \bPhi^\intercal J_h^\intercal(\bPhi\bPhi^\intercal{\bx})\mathbf{c}] -  \bB_r^\intercal [\bPhi^\intercal \bQ \bPhi\bPhi^\intercal{\bx}+ \bPhi^\intercal J_h^\intercal(\bPhi\bPhi^\intercal{\bx}) \mathbb{P}^\intercal \mathbf{c}] \\
        &+\bB_r^\intercal [\bPhi^\intercal \bQ \bPhi\bPhi^\intercal{\bx}+ \bPhi^\intercal J_h^\intercal(\bPhi\bPhi^\intercal{\bx}) \mathbb{P}^\intercal \mathbf{c}] - \bB_r^\intercal [\bPhi^\intercal \bQ \bPhi\bx_r + \bPhi^\intercal J_h^\intercal(\bPhi \bx_r) \mathbb{P}^\intercal \mathbf{c}]. 
\end{aligned}
\label{eq:d_ey00_DEIM}
\end{equation}
Defining 
$$\boldsymbol{\varphi} \coloneqq {\by}-\bB_r^\intercal \bPhi^\intercal \nabla_{\bx} H(\bx),\quad
\boldsymbol{\rho}\coloneqq \bx-\bPhi\bPhi^\intercal \bx,\quad\boldsymbol{\xi} \coloneqq (\bI-\mathbb{P})J_h(\bPhi\bPhi^\intercal\bx) \bPhi,\quad  \boldsymbol{\theta} \coloneqq \bPhi\bPhi^\intercal \bx-\bPhi \bx_r,$$
we then have 
\begin{equation*}
\|\be_{\by}\| \leq \|\boldsymbol{\varphi}\|\,+ C_4 \|\boldsymbol{\rho}\| + C_5 \|\boldsymbol{\xi}\|+ C_6 \|\boldsymbol{\theta}\|,
\label{eq:d_ey1_DEIM}
\end{equation*}
for any $t \in [0,T]$.
Hence,
$$
\|\be_{\by}\|^2 \leq 4 \left(C_4^2 \|\boldsymbol{\rho}\|^2 + C_6^2 \|\boldsymbol{\theta} \|^2 + \|\boldsymbol{\varphi}\|^2 + C_5^2 \|\boldsymbol{\xi}\|^2 \right).
$$
After integrating both sides from $0$ to $T$, using
$\be_x = \boldsymbol{\rho}+\boldsymbol{\theta}$ and the orthogonality of $\boldsymbol{\rho}$ and $\boldsymbol{\theta}$, we get 
$$
\int_0^T \|\be_{\by}\|^2\, { \rm d} t \leq 4 \left( \text{max}(C_4^2,C_6^2)\int_0^T \|{\be_x}\|^2\, { \rm d} t +
 \int_0^T\|\boldsymbol{\psi}\|^2\,{\rm d} t + C_5^2  \int_0^T \|\boldsymbol{\xi}\|^2\, { \rm d} t\right).
$$
This proves the theorem.
\end{proof}

When hyper-reduction is not applied, the \text{pH-OpInf-DEIM} model coincides with the \text{pH-OpInf-ROM}, and we have the following result. 
\begin{corollary}
    Let $\bx(t)$ and $\by(t)$ be the state and output of the FOM \eqref{eq:PH_system}, respectively, on the time interval $[0, T]$ and ${\bx_{r}}(t)$ and ${\by_{r}}(t)$ be the reduced-order state and output, respectively, of the \text{pH-OpInf-ROM} defined in \eqref{eq:PH_rom_inf_system} over the same interval.  
Let $\bD_r = \bJ_r-\bR_r$ in \eqref{eq:PH_rom_inf_system} and suppose that $\nabla_{\bx} H(\bx)$ and $J_h(\bx)$ are both Lipschitz continuous, then the ROM state approximation error satisfies
\begin{equation}
\label{eq:c1_r1}
\begin{aligned}
\int_0^T \|\bx-\bPhi{\bx_r}\|^2\, { \rm d} t 
\leq 
 \widehat C(T) \ \bigg( 
\underbrace{\int_0^T \|\bx-\bPhi\bPhi^\intercal \bx\|^2\, { \rm d}t}_{\text{ \em projection error}} 
&+ 
\underbrace{\int_0^T \|\dt{\bx}-\mathcal{D}_t[\bx]\|^2\, { \rm d}t}_{\text{\em data error}}
\\
&+ 
\underbrace{\int_0^T \|\bPhi^\intercal \mathcal{D}_t[\bx]  - \bD_r \bPhi^\intercal \nabla_{\bx} H(\bx) - \bB_r \bu(t)\|^2\, { \rm d}t}_{\text{\em optimization error}} 
\bigg),
\end{aligned}
\end{equation}
where
the constant $\widehat C(T) = \max\{1+C_8^2T \alpha(T),T \alpha(T)\} $ with 
$$C_7= \mathcal{C}_{\mathrm{log-Lip}}[\bPhi \bD_r \bPhi^\intercal\nabla_{\bx}H],\quad
C_8= \|\bD_r \bPhi^\intercal\|\, \mathcal{C}_{\mathrm{Lip}}[\nabla_{\bx}H],\quad {\rm and } \quad
\alpha(T)= 3\int_0^T e^{2C_7 (T-\tau)}\,{ \rm d} \tau.$$
The corresponding output approximation error satisfies
\begin{equation}
\label{eq:c1_r2}
\int_0^T \|\by-{\by_r}\|^2\, { \rm d} t 
\leq  
\widehat C \ \bigg( 
\underbrace{\int_0^T \|\bx-\bPhi \bx_r\|^2\, { \rm d}t}_{\text{ \em state approximation error}}  
+ 
\underbrace{\int_0^T \|{\by}-\bB_r^\intercal \bPhi^\intercal \nabla_{\bx} H(\bx)\|^2\, { \rm d}t}_{\text{\em optimization error}} \bigg),
\end{equation}
where the constant $\widehat C =2 \max \{1, C_9^2\}$ with $C_9 = \|\bB_r^\intercal \bPhi^\intercal\|\mathcal{C}_{\mathrm{Lip}}[\nabla_{\bx}H]\,$.
\label{corollary1}
\end{corollary}
\begin{proof}
    Following the same steps outlined in Theorem~\ref{thm:1} but replacing inequalities \eqref{eq:err_DEIM1}, \eqref{eq:err_DEIM2}, and \eqref{eq:err_DEIM3} with the inequality
    \begin{equation*}
     \left<\boldsymbol{\theta}, \bPhi \bD_r \bPhi^\intercal\, \left(\nabla_{\bx} H(\bPhi\bPhi^\intercal\bx) - \nabla_{\bx} H(\bPhi\bx_r)\right) \right>
    \leq \|\boldsymbol{\theta}\|\, \mathcal{C}_{\mathrm{Log-Lip}}[\bPhi \bD_r \bPhi^\intercal\nabla_{\bx}H]\, \|\boldsymbol{\rho}\|,
    \label{eq:errx_without_DEIM}
    \end{equation*}
    we obtain the state approximation error in \eqref{eq:c1_r1}. 
    Then using the same argument as in Theorem~\ref{thm:2} and changing the equation \eqref{eq:d_ey00_DEIM} to  
    \begin{equation*}
\begin{aligned}
{\be_{\by}} &= \by- \bB_r^\intercal \nabla_{\bx_r} H_{r}(\bx_r) = \by  - {\bB_r}^\intercal {\bPhi}^\intercal \nabla_{\bx} H(\bx) 
        +\bB_r^\intercal \bPhi^\intercal \nabla_{\bx} H(\bx) - \bB_r^\intercal\bPhi^\intercal  \nabla_{{\bx}} H(\bPhi{\bx_r}),
\end{aligned}
\label{eq:d_ey00}
\end{equation*}
we obtain the inequality
\begin{equation*}
\|\be_{\by}\| \leq \|\boldsymbol{\varphi}\|\,+  \|\bB_r^\intercal \bPhi^\intercal\|\mathcal{C}_{\mathrm{Lip}} [\nabla_{\bx}H]\|\be_\bx\|,
\label{eq:d_ey1}
\end{equation*} 
and further the output approximation error in \eqref{eq:c1_r2}.
This completes the proof. 
\end{proof}

\begin{remark}
The re-projection method has been introduced in \cite{peherstorfer2020sampling} to reduce optimization errors for operator inference. 
When the FOM drift terms are polynomial and the time derivative data is discretized in the same manner as the FOM, the re-projection approach ensures that the inferred ROM recovers the intrusive Galerkin-projection ROM and the optimization error is zero. However, this does not hold for port-Hamiltonian systems. Suppose the re-projection method generates a sequence of approximate state vectors $\{\widetilde{\bx}_0, \widetilde{\bx}_1, \ldots, \widetilde{\bx}_s\}$, in which $\widetilde{\bx}_{n+1}$ is obtained from \eqref{eq:PH_system1} after approximating the time derivative by $\mathcal{D}_t$ and taking $\widetilde{\bx}_n$ as the initial data. Then it satisfies   
\begin{equation}\label{eq:reproj_1}
     \mathcal{D}_t[\bPhi^\intercal\,\widetilde{\bx}_n]=\bPhi^\intercal \bD \nabla_{\bx} H(\bPhi \bPhi^\intercal\, \widetilde{\bx}_n) +\bPhi^\intercal \bB \bu_n.
\end{equation}
Define $\widetilde{\mathbf{X}} \coloneqq \bPhi\bPhi^\intercal \left[\widetilde{\bx}_0, \widetilde{\bx}_1, \ldots, \widetilde{\bx}_s \right] \in \mathbb{R}^{r\times (s+1)}$. Assuming that $\bB_r = \bPhi^\intercal \bB$ has been obtained, we next find $\bD_r\in \mathbb{R}^{r\times r}$ by minimizing  
$$
\frac{1}{2} \left\|\bPhi^\intercal \,  \mathcal{D}_t\left[\widetilde{\bX}\right] - \bD_r \bPhi^\intercal \nabla_{\bx}H\left(\widetilde{\bX}\right) - \bB_r \bU \right\|_F^2 \quad \text{s.t.} \quad \frac{1}{2}(\bD_r+\bD_r^\intercal) \preccurlyeq \mathbf{0}.
$$
Because of the relation in \eqref{eq:reproj_1} and $\mathcal{D}_t[\bPhi^\intercal\,\widetilde{\bX}] = 
\bPhi^\intercal\, \mathcal{D}_t[\widetilde{\bX}]$, it is equivalent to minimize 
$$\frac{1}{2} \left\|\bPhi^\intercal \bD \nabla_{\bx} H\left(\widetilde{\bX}\right) - \bD_r \bPhi^\intercal \nabla_{\bx}H\left(\widetilde{\bX}\right)  \right\|_F^2 
\quad \text{s.t.} \quad \frac{1}{2}(\bD_r+\bD_r^\intercal) \preccurlyeq \mathbf{0}.
$$ 
In general, the minimizer $\bD_r$ solved from the optimization can not reduce the objective function to zero. 
\end{remark}

\section{Numerical Experiments} \label{sec: Num_exp}
In this section, we provide two numerical experiments to demonstrate the effectiveness and qualitative features of the proposed
method. In Section \ref{sec:num_errors} we define the error metrics used to evaluate the quality and accuracy of the \text{pH-OpInf-ROM} and the \text{pH-OpInf-DEIM} model. 
A \textit{linear} mass-spring-damper system \cite{polyuga2010model,morandin2023port} is studied in Section \ref{sec:linear case}. As the system lacks a nonlinearity, the \text{pH-OpInf-ROM} will be applied and investigated. The \textit{nonlinear} Toda lattice model \cite{chaturantabut2016structure}, which describes a chain of particles with exponential interactions between neighboring particles, is tested in Section \ref{sec:nonlinear case}. Due to nonlinearities in this system, the \text{pH-OpInf-DEIM} model will be employed and tested.

\subsection{Error Measures}\label{sec:num_errors}
To test the ROMs' accuracy, we compute the reduced-order approximation errors for both state variable $\bx$ and output $\by$, as functions of the dimension $r$ of the ROMs: 
\begin{equation}
    \mathcal{E}_\bx(r) := \sqrt{\frac{T}{N} \sum_{k=1}^N \left\|\bx(t^k)-\bPhi\bx_r(t^k)\right\|^2}, 
   \label{eq:e_x_approx}
\end{equation}
\begin{equation}
    \mathcal{E}_\by(r) := \sqrt{\frac{T}{N} \sum_{k=1}^N \left\|\by(t^k)-\by_r(t^k)\right\|^2}, 
   \label{eq:e_y_approx}
\end{equation}
where $\bx(t^k)$, $\by(t^k)$ are full-order solutions and $\bx_r(t^k)$, $\by_r(t^k)$ are reduced-order solutions, for $k = 1,2,\dots, N$, and $\bPhi$ is the POD basis matrix. 
To further illustrate the error analysis in \eqref{eq:error_x_bound} and \eqref{eq:error_y_bound}, we evaluate the \textit{projection error}, \textit{optimization error} and \textit{hyper-reduction error} as 
\begin{align}
\mathcal{E}_{\text{proj}_\bx}(r) & := \sqrt{\frac{T}{N} \sum_{k=1}^N \left\|\bx(t^k)-\bPhi\bPhi^\intercal \bx(t^k)\right\|^2}, \label{eq:e_proj_x} \\
\mathcal{E}_{\text{proj}_{\nabla H}}(r) & := \sqrt{\frac{T}{N} \sum_{k=1}^N \left\|\nabla_{\bx} H(\bx(t^k))-\bPhi\bPhi^\intercal \nabla_{\bx} H(\bx(t^k))\right\|^2}, \label{eq:e_proj_y} \\
\mathcal{E}_{\text{opt}_\bx}(r) & := \sqrt{\frac{T}{N} \sum_{k=1}^N \left\|\bPhi^\intercal \mathcal{D}_t[\bx(t^k)]-( \bJ_r - \bR_r) \bPhi^\intercal \nabla_{\bx} H(\bx(t^k)) - \bB_r \bu(t^k)\right\|^2},\label{eq:e_opt_x} \\
\mathcal{E}_{\text{opt}_\by}(r) &:= \sqrt{\frac{T}{N} \sum_{k=1}^N \left\| \by(t^k) - \bB_r^\intercal \bPhi^\intercal \nabla_\bx H(\bx(t^k))\right\|^2},
\label{eq:e_opt_y}
\\
\mathcal{E}_{\text{DEIM}}(r) &:= \sqrt{\frac{T}{N} \sum_{k=1}^N \left\| (\mathbf{I}-\mathbb{P})J_h(\bPhi \bPhi^\intercal \bx(t^k))\bPhi \right\|^2}.
\label{eq:e_DEIM}
\end{align}
In numerical tests, the FOM solutions are simulated in the time interval $[0,T_{\text{FOM}}]$, while $[0,T_{\text{ROM}}]$ in the ROM simulations. 
Furthermore, we utilize the optimization software MOSEK  \cite{aps2019mosek}, under an academic license, to solve the optimization problems via the open-source CVXPY library \cite{diamond2016cvxpy,agrawal2018rewriting}. 

\subsection{Linear Case: Mass-Spring-Damper System}\label{sec:linear case}
\begin{figure}[!ht]
    \centering
\includegraphics[width=0.8\linewidth]{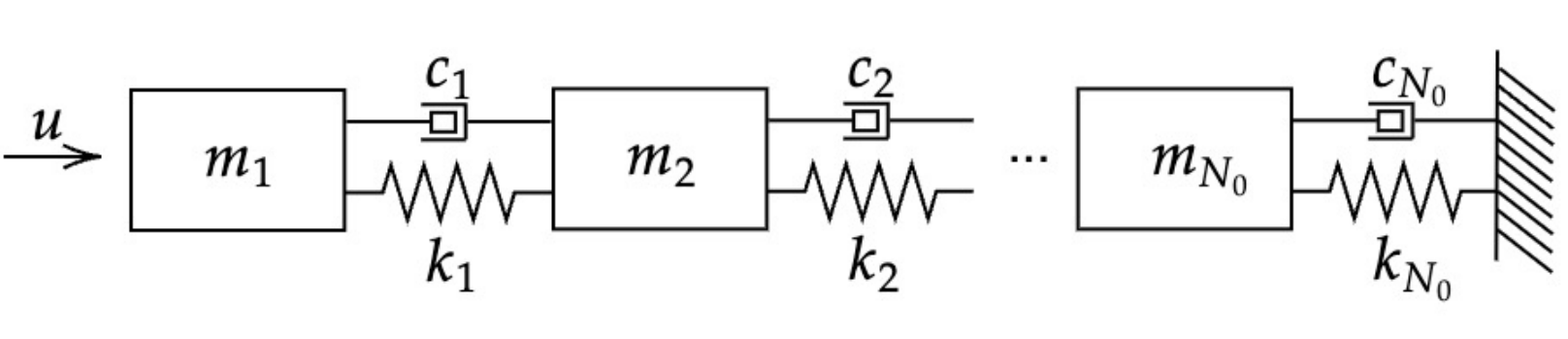}
    \caption{The Mass-Spring-Damper System.}
    \label{fig:msd}
\end{figure}
In our first experiment, we consider the linear mass-spring-damper system, illustrated in Figure \ref{fig:msd},
consisting of masses $m_i$, springs with constants $k_i$ and dampers with non-negative damping constants $c_i$ for $i=1,2,\dots, N_0$. 
The input $u(t)$ is the external force exerted on the first mass $m_1$ and the output $y(t)$ is the velocity of the first mass. The state vector $\bx=[q_1,p_1,\dots,q_{N_0},p_{N_0}]\in \mathbb{R}^{2N_0}$, where $q_i$ and $p_i$ represent the displacement and momentum of the $i$th mass, respectively. The system is governed by \eqref{eq:PH_system} with Hamiltonian $H(\bx) =  \frac{1}{2} \bx^\intercal \bQ \bx$. For instance, when $N_0=3$, we have 
$$
\mathbf{J}= 
 \begin{bmatrix}
 0&1&0&0&0&0 \\
 -1&0&0&0&0&0   \\
 0&0&0&1&0&0 \\
 0&0&-1&0&0&0 \\
 0&0&0&0&0&1 \\
 0&0&0&0&-1&0 \\
\end{bmatrix}
, \quad
\mathbf{Q}= 
 \begin{bmatrix}
k_1&0&k_1&0&0&0\\
0&\frac{1}{m_1}&0&0&0&0\\
-k_1&0&k_1+k_2&0&-k_2&0\\
0&0&0&\frac{1}{m_2}&0&0\\
0&0&-k_2&0&k_2+k_3&0\\
0&0&0&0&0&\frac{1}{m_3}
\end{bmatrix},
$$
$$
\mathbf{R}= 
 \begin{array}{cc}
\text{diag}(0,c_1,0,c_2,0,c_3),
\end{array}
\quad
\text{and } 
\quad
\mathbf{B}^\intercal = 
\begin{bmatrix}
    0&1&0&0&0&0
\end{bmatrix}.
$$ 
\paragraph{Problem and computational setting}
We set $\bx_0=\mathbf{0}$, $N_0=100$, $m_i=4, k_i=4$ and $c_i=1$ for $i=1,2,\dots,100$, and choose the input of the system to be $u(t)= \exp({-\frac{t}{2}})\sin(t^2)$. For the full-order simulation, we set $T_{\text{FOM}}=10$ and use the implicit midpoint rule for time integration with a step size $\Delta t=0.01$.  

After constructing the snapshot matrices $\bX$, $\bY$, and $\bU$, and generating the reduced basis matrix $\bPhi \in \mathbb{R}^{n \times r}$, which consists of the $r$ leading left singular vectors of $\bX$, we apply \text{pH-OpInf-W} and \text{pH-OpInf-R} to infer $\bJ_r, \bR_r$ and $\bB_r$, as described in Section \ref{sec: GP-OpInf}. Since selecting appropriate hyper-parameters, $\lambda_W$ and $\lambda_R$, is crucial for the optimization, we next investigate the effects of these optimization parameters. 

\subsubsection*{Test 1. Effects of optimization parameters}\label{sec:parameter test}
First, we study the influence of the weight parameter $\lambda_W$ in the optimization problem \eqref{eq:opt} on the quality of the  \text{pH-OpInf-W}. 
Note that, in \eqref{eq:opt}, $\lambda_W$ serves as a balancing factor between the two error components, $\mathcal{E}_{\text{opt}_\bx}$ and $\mathcal{E}_{\text{opt}_\by}$, in the minimization process. 
We vary the reduced dimension $r$ from $5$, $10$, $15$ to $20$ and test eight values for $\lambda_W$ from the set $\{10^0,10^1,\dots,10^7\}$, for each value of $r$. 
Figure \ref{fig:linear_W_lambda_test} shows the corresponding optimization errors. 
As $\lambda_W$ increases up to $10^5$, the optimization error $\mathcal{E}_{\text{opt}_\bx}$ remains nearly unchanged for all $r$, while $\mathcal{E}_{\text{opt}_\by}$ decreases. Beyond this point, $\mathcal{E}_{\text{opt}_\by}$ either decays or levels off, while $\mathcal{E}_{\text{opt}_\bx}$ remains nearly unchanged for $r=5, 10, 15$, but increases for $r=20$. 
Therefore, $\lambda_W=10^5$ is selected in the rest of the experiments. 
\begin{figure}[!ht]
    \centering
    \includegraphics[width=0.45\linewidth]
    {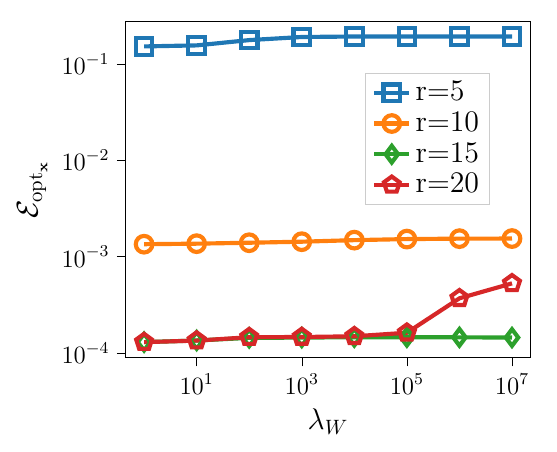}\hspace{0.3cm}
    \includegraphics[width=0.45\linewidth]{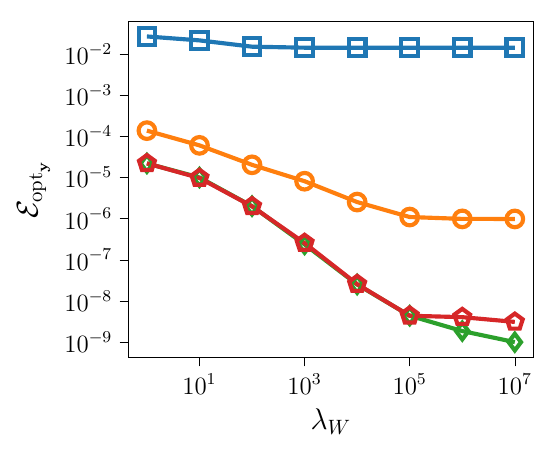}
    \caption{Mass-Spring-Damper System: Optimization errors by \text{pH-OpInf-W} of $r$ dimensions using different values of  $\lambda_W$: $\mathcal{E}_{\text{opt}_\bx}$ (left) and $\mathcal{E}_{\text{opt}_{\by}}$ (right).}
    \label{fig:linear_W_lambda_test}
\end{figure}

Next, we test the influence of the regularization parameter $\lambda_R$ on the optimization problem and the resulting effects on the \text{pH-OpInf-R} ROM. When $\lambda_R = 0$, we find that the solution $\bB_r$ from equation \eqref{eq:opt1} using least squares is poorly conditioned, leading to unstable ROM simulations. 
Taking $r = 5, 10, 15$ and $20$ and selecting eight values of $\lambda_R$ from the set $\{ 10^{-14}, 10^{-13}, \dots, 10^{-7} \}$, we solve \eqref{eq:opt1} and \eqref{eq:opt2}, respectively. The optimization errors $\mathcal{E}_{\text{opt}_\bx}$ and $\mathcal{E}_{\text{opt}_\by}$ are shown in Figure \ref{fig:linear_R_lambda_test}.  
For $r= 5$ and $10$, the optimization errors are relatively steady as $\lambda_R$ varies. However, for larger values of $r = 15$ and $20$, as $\lambda_R$ increases, both $\mathcal{E}_{\text{opt}_\bx}$ and $\mathcal{E}_{\text{opt}_\by}$ initially decrease (with $\mathcal{E}_{\text{opt}_\bx}$ remaining unchanged for $r=15$), but after $\lambda_R = 10^{-11}$, both begin to increase. 
Therefore, in subsequent experiments, we select $\lambda_R=10^{-11}$ in \text{pH-OpInf-R}. 
\begin{figure}[!ht]
    \centering
    \includegraphics[width=0.45\linewidth]{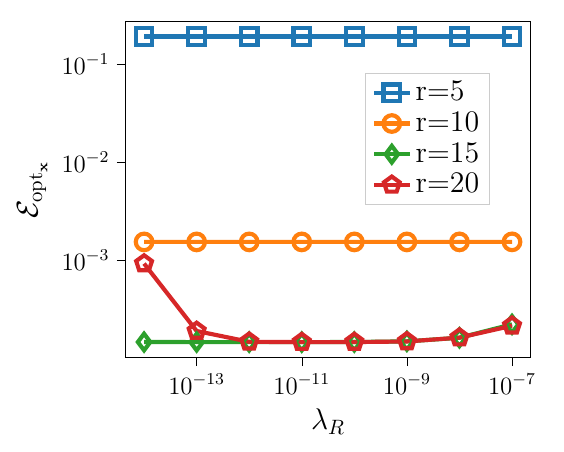}\hspace{0.3cm}
    \includegraphics[width=0.45\linewidth]{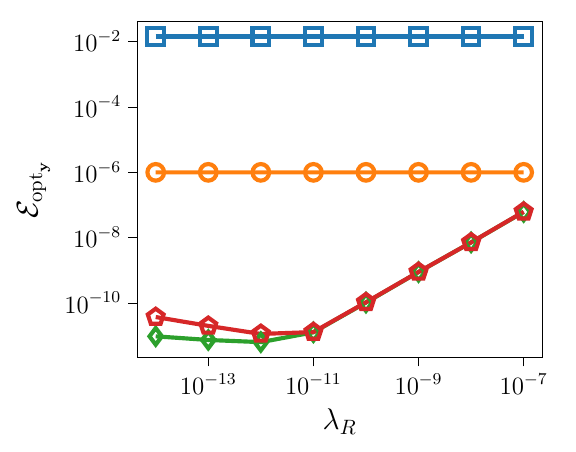}
    \caption{Mass-Spring-Damper System: Optimization errors by \text{pH-OpInf-R} of $r$ dimensions using different values of $\lambda_R$: $\mathcal{E}_{\text{opt}_\bx}$ (left)  and $\mathcal{E}_{\text{opt}_\by}$ (right).}
    \label{fig:linear_R_lambda_test}
\end{figure}

For \text{pH-OpInf-W} with $\lambda_W=10^5$ and \text{pH-OpInf-R} with $\lambda_R=10^{-11}$, the minimum eigenvalues of the obtained $\bR_r$ are listed in Table \ref{tab:linear_min_eigvalue} for $r = 5, 10, 15$ and $20$, which demonstrates the positive semi-definiteness of the obtained $\bR_r$. 
\begin{table}[!ht]
\centering
\caption{Mass-Spring-Damper System: Minimum eigenvalues of $\bR_r$ obtained by \text{pH-OpInf-W} and \text{pH-OpInf-R} for different values of $r$.}
\begin{tabular}{c|c|c|c|c}
   & r=5        & r=10       & r=15       & r=20       \\ \hline
\text{pH-OpInf-W}       & 1.460 $\times 10^{-5}$  & 7.864$\times 10^{-7}$  & 7.729$\times 10^{-5}$  & 9.396$\times 10^{-5}$  \\ 
\text{pH-OpInf-R}        & 5.486$\times 10^{-8}$ & 2.807$\times 10^{-6}$ & 3.567$\times 10^{-5}$ & 4.682$\times 10^{-5}$ 
\end{tabular}
\label{tab:linear_min_eigvalue}
\end{table}

\subsubsection*{Test 2. Illustration of the error estimates} 
Based on Corollary \ref{corollary1}, for this linear system, the state error is bounded by the sum of projection error, data error, and optimization error; and the output error is bounded by the state approximation error and the corresponding optimization error. 
Next, we illustrate the error estimates for the \text{pH-OpInf-ROM} \eqref{eq:PH_rom_inf_system}, generated using the inferred reduced operators from either \text{pH-OpInf-W} or \text{pH-OpInf-R}, and compare the approximation errors with those of the intrusive \text{SP-G-ROM} of the same dimensions.
To this end, we choose $\Delta t=10^{-3}$ such that the data error is negligible relative to other sources of error.

\begin{figure}[!ht]
\centering
  \includegraphics[width=0.45\linewidth]{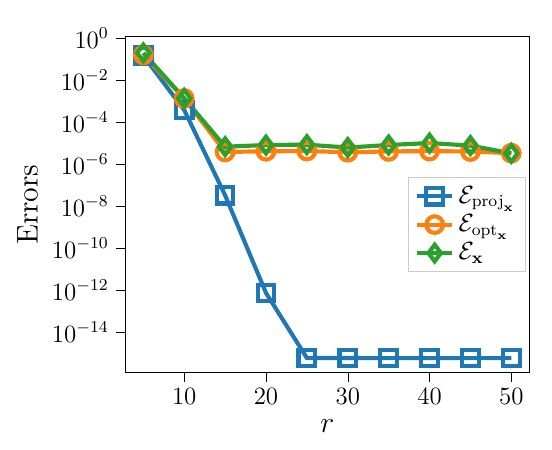} \hspace{0.3cm}
   \includegraphics[width=0.45\linewidth]{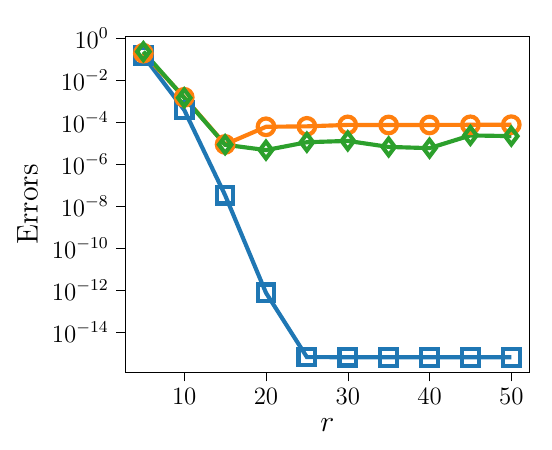} 
  \includegraphics[width=0.45\linewidth]{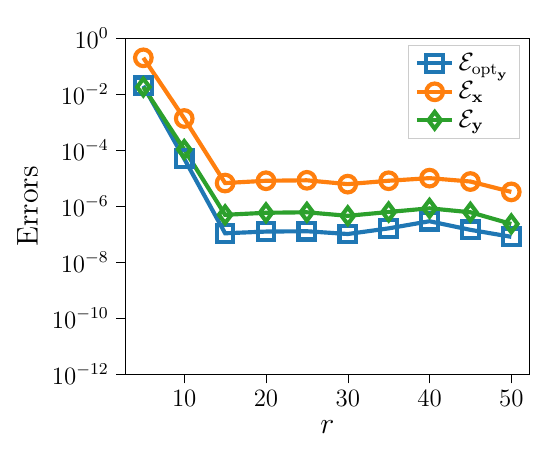} \hspace{0.3cm}
   \includegraphics[width=0.45\linewidth]{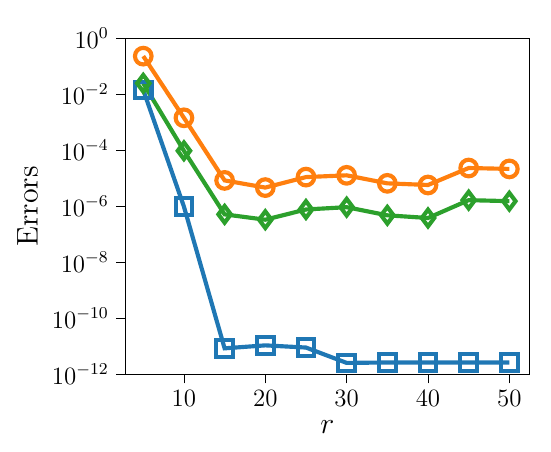} 
  \includegraphics[width=0.45\linewidth]{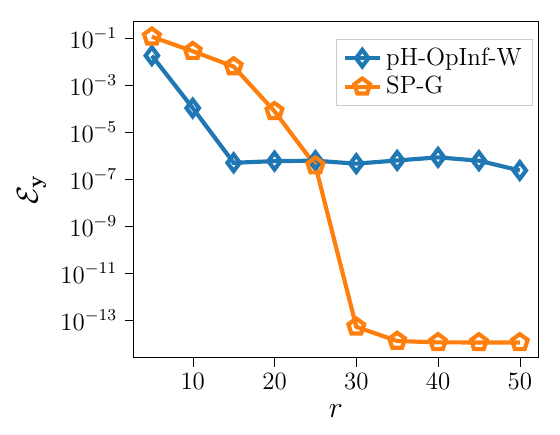} \hspace{0.3cm}
  \includegraphics[width=0.45\linewidth]{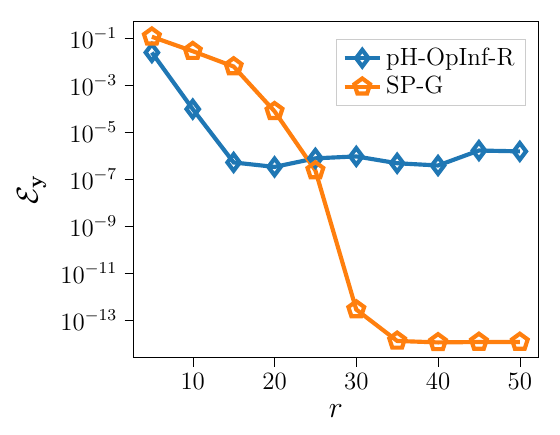} 
\caption{Mass-Spring-Damper System: Numerical errors of \text{pH-OpInf-W} (left) and \text{pH-OpInf-R} (right) of $r$ dimensions when $T_{\text{FOM}}=T_{\text{ROM}}=10$.}
\label{fig:linear_error_estimate}
\end{figure}

Setting $T_{\text{FOM}}=T_{\text{ROM}}=10$ and gradually increasing the ROM dimension $r$ by $5$, from $5$ to $50$, we simulate the \text{pH-OpInf-ROMs}. 
In this case, the singular values of the snapshot matrix $\bX$ decay quickly: the first 5 POD modes capture 95.24\% of the snapshot energy, while the first 10 modes capture 99.99\%.
The associated numerical errors defined in \eqref{eq:e_x_approx}--\eqref{eq:e_opt_y} are plotted in Figure \ref{fig:linear_error_estimate}: the left column corresponds to \text{pH-OpInf-W}, and the right column corresponds to \text{pH-OpInf-R}. 
Note that the magnitude of the state vectors is 0.661, while that of the output is 0.083. 
From the top row, we observe that in both cases, the projection error $\mathcal{E}_{\text{proj}_\bx}$ continues to decay as $r$ increases, eventually reaching the machine precision beyond $r=25$. Meanwhile, the optimization error $\mathcal{E}_{\text{opt}_\bx}$ decreases to around $10^{-5}$ as $r$ increases to $20$ and remains nearly steady. Consequently, the state approximation error $\mathcal{E}_\bx$ is dominated by $\mathcal{E}_{\text{opt}_\bx}$ and exhibits similar behavior. 
From the middle row, we find that the optimization error $\mathcal{E}_{\text{opt}_{\by}}$ from \text{pH-OpInf-R} is smaller than that from \text{pH-OpInf-W} when $r>5$, but in both cases, $\mathcal{E}_{\text{opt}_{\by}}$ is much smaller than the state approximation error $\mathcal{E}_\bx$. Thus, the output approximation error $\mathcal{E}_\by$ is dominated by $\mathcal{E}_\bx$ and displays similar behavior.  
In the bottom row, $\mathcal{E}_\by$ from the \text{pH-OpInf-ROMs} is compared with that of the \text{SP-G-ROM}. 
The \text{pH-OpInf-ROM}, based on either \text{pH-OpInf-W} or \text{pH-OpInf-R}, attains better results when $r<25$. Beyond that, \text{SP-G-ROM} achieves better approximations. 
Therefore, in case when the computation budget is a concern, a ROM with a lower dimension is preferable, and the \text{pH-OpInf-ROM} is the better choice. 
Furthermore, the time evolution of the output and approximate Hamiltonian from the \text{pH-OpInf-ROMs} is shown in 
Figure \ref{fig:linear_ROM_energy_Tf10}. We observe that the reduced-order approximations of both the output and the Hamiltonian, 
using either \text{pH-OpInf-W} or \text{pH-OpInf-R}, closely match those from the FOM simulation when $r\geq 10$. 
\begin{figure}[!ht]
\centering
    \includegraphics[width=0.45\linewidth]{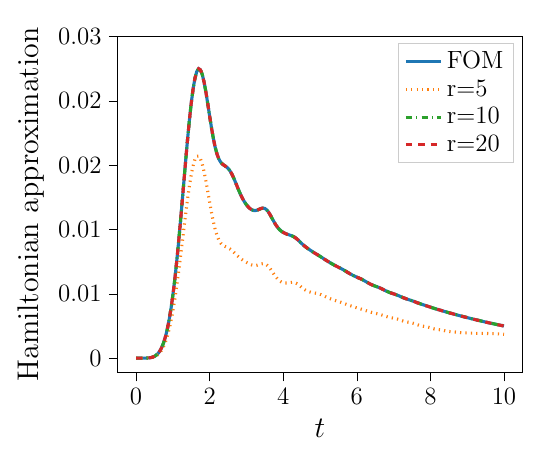}\hspace{0.3cm}
    \includegraphics[width=0.45\linewidth]{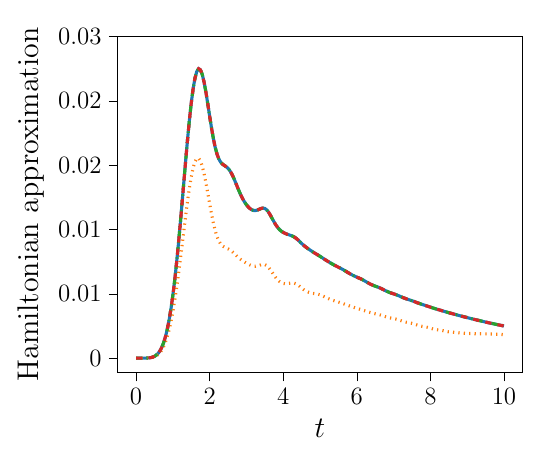}
    \includegraphics[width=0.45\linewidth]{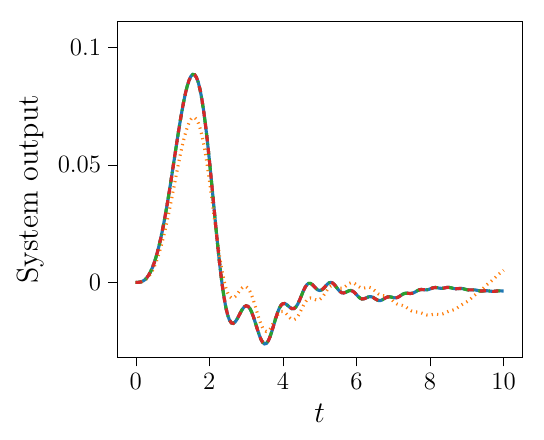}\hspace{0.3cm}
    \includegraphics[width=0.45\linewidth]{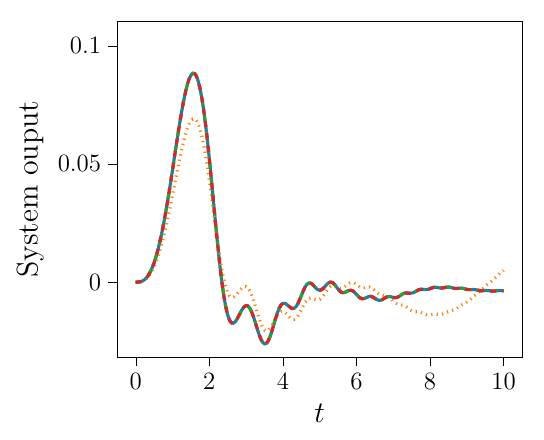}
    \caption{Mass-Spring-Damper System: Hamiltonian approximation (top row) and system output (bottom row) from \text{pH-OpInf-W} (left) and \text{pH-OpInf-R} (right) of $r$ dimensions, along with those of the FOM. All subplots share the same legend shown in the first one.}
    \label{fig:linear_ROM_energy_Tf10}
\end{figure}

Overall, the \text{pH-OpInf-ROM} constructed using reduced operators from either optimization problem provides good approximations. However, compared to \text{pH-OpInf-W}, the \text{pH-OpInf-R} achieves a smaller optimization error for the output. Therefore, we use \text{pH-OpInf-R} in the next experiment.

\subsubsection*{Test 3. Performance at a different input}\label{sec:linear_train_test}
We test the \text{pH-OpInf-ROM} using an input different from the one used to generate the training snapshots.
\begin{figure}[!ht]
    \centering
    \includegraphics[width=0.45\linewidth]{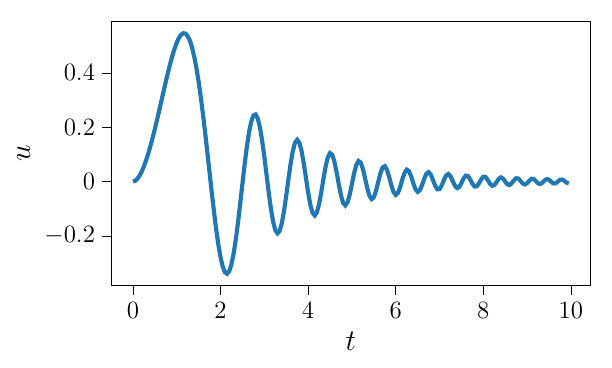}\hspace{0.3cm}
    \includegraphics[width=0.45\linewidth]{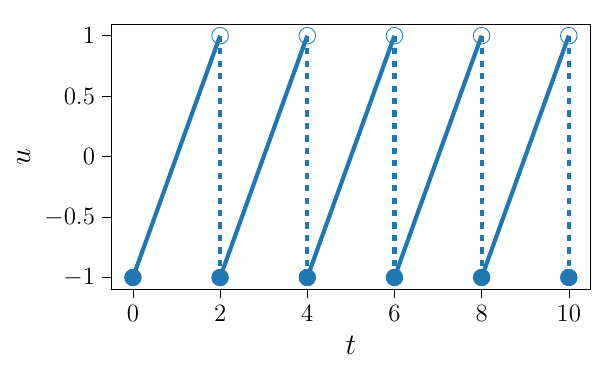}
    \caption{Mass-Spring-Damper System: Training input (left) and testing input (right) for Test 3.}
    \label{fig:linear_input}
\end{figure}
To infer the reduced operator, we use the data associated with the training input $u(t)= \exp({-\frac{t}{2}})\sin(t^2)$ as shown in Figure \ref{fig:linear_input} (left), which has increasing frequency to excite the model. 
After building the ROM with $r=20$ dimensions, we test it using a different input $u$, shown in Figure \ref{fig:linear_input} (right). 
The sawtooth signal is a challenging test case, as it is discontinuous. The corresponding output and Hamiltonian function from the FOM and ROM simulations are shown in Figure \ref{fig:linear_ROM_solu_energy}. The results indicate that the \text{pH-OpInf-ROM} performs well, as both the reduced-order output and the Hamiltonian function closely match the FOM results. 
\begin{figure}[!ht]
    \centering
    \includegraphics[width=0.44\linewidth]{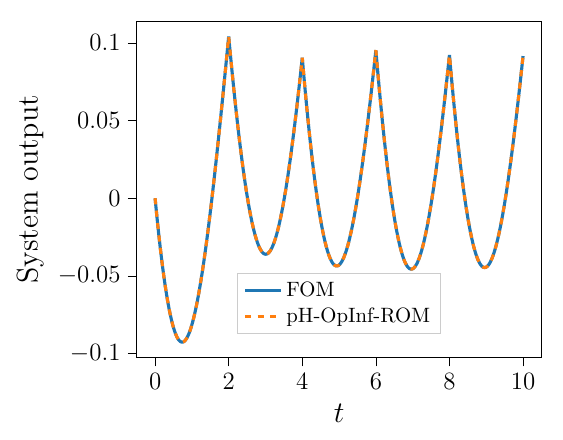}\hspace{0.3cm}
    \includegraphics[width=0.425\linewidth]{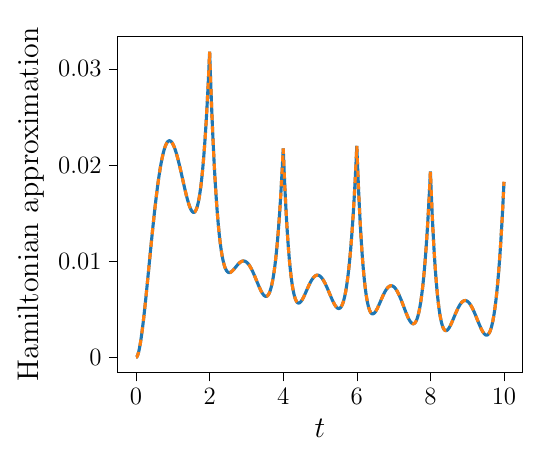}
    \caption{Mass-Spring-Damper System: Simulation results of \text{pH-OpInf-ROM} for $r=20$ and the FOM: system output (left) and Hamiltonian approximation (right) for Test 3.}
    \label{fig:linear_ROM_solu_energy}
\end{figure}

\subsection{Nonlinear Case: Toda Lattice Model} \label{sec:nonlinear case}
We consider the nonlinear Toda lattice model, which describes the motion of a chain of particles, each one connects to its nearest neighbors with `exponential springs'. The equations of motion for the $N_0$-particle Toda lattice with such exponential interactions can be written in the form of a nonlinear pH system  as in \eqref{eq:PH_system} with
$$
\mathbf{J}= 
\left[\begin{array}{cc}
\mathbf{0} & \mathbf{I}  \\
-\mathbf{I}& \mathbf{0} 
\end{array} \right]_{n\times n}
, \quad
\mathbf{R} = 
\left[\begin{array}{cc}
\mathbf{0} & \mathbf{0}  \\
\mathbf{0}& \text{diag}(\gamma_1, \dots,\gamma_{N_0}) 
\end{array} \right]_{n\times n}, 
\quad
\text{and } 
\quad
\mathbf{B} = 
\left[\begin{array}{c}
\mathbf{0} \\
\mathbf{e_1}
\end{array} \right]_{n \times 1},
$$ 
where $\be_1=[1,0,\dots,0]^\intercal_{N_0}$, $n=2N_0$, $\gamma_j$ represents damping coefficients associated with the $j$th particle in the system and 
$$
\mathbf{x} = 
\left[\begin{array}{c}
\mathbf{q} \\
\mathbf{p}
\end{array} \right]_{n \times 1}
\,\,
\text{, } 
\,\,
\mathbf{q} = \left[q_1,\dots,q_{N_0}\right]^\intercal 
\,\,
\text{and } 
\,\,
\mathbf{p} = \left[p_1,\dots,p_{N_0}\right]^\intercal,
$$
with $q_j$ and $p_j$, for $j=1,\dots,N_0$, being the displacement of the $j$th particle from its equilibrium position and the momentum, respectively. The associated Hamiltonian function is nonlinear, given by 
\begin{equation}\label{eq:Toda energy}
    H(\bx) =\sum_{k=1}^{N_0} \frac{1}{2} p_k^2 + \sum_{k=1}^{N_0-1} \exp(q_  k-q_{k+1}) + \exp(q_{N_0}) -q_1-N_0.
\end{equation}

\paragraph{Problem and computational setting} 
We set $\bx_0 = \mathbf{0}$, $N_0 =1000$, and $\gamma_j=0.1$ for $j=1,\dots,N_0$.
The system is excited by the input $u(t) = 0.1\sin(t)$. 
In the full-order simulation with the final time $T_{\text{FOM}}=50$, we use the implicit midpoint rule with the time step size $\Delta t= 0.01$ for time integration. 
The reduced operators $\bJ_r, \bR_r$ and $\bB_r$ are then inferred as described in Section \ref{sec: GP-OpInf} through either \text{pH-OpInf-W} or \text{pH-OpInf-R}. Next, we investigate the effects of the optimization parameters $\lambda_W$ and $\lambda_R$ in the optimization problems.

\subsubsection*{Test 1. Effects of optimization parameters} \label{sec:hyper-parameters}
We analyze the effect of the weight parameter $\lambda_W$ in the optimization problem \eqref{eq:opt} on the quality of the \text{pH-OpInf-W} ROM, by taking reduced dimensions $r=20, 40, 60$ and $80$ and testing eight values of $\lambda_W$ from the set $\{10^0,10^1,\dots,10^7\}$. Figure \ref{fig:nonlinear_W_lambba_test} illustrates the optimization errors $\mathcal{E}_{\text{opt}_\bx}$ (left) and $\mathcal{E}_{\text{opt}_{\by}}$ (right) for varying $\lambda_W$ and $r$. When $r$ is small, variations in $\lambda_W$ have a minimal impact on the optimization. However, for larger dimensions ($r = 60$ and $80$), increasing $\lambda_W$ reduces $\mathcal{E}_{\text{opt}_{\by}}$ while increasing $\mathcal{E}_{\text{opt}_\bx}$ in the optimization problem \eqref{eq:opt}. To balance these terms, we set $\lambda_W=10^3$ for the remainder of the experiments. 
\begin{figure}[!ht]
    \centering
    \includegraphics[width=0.45\linewidth]{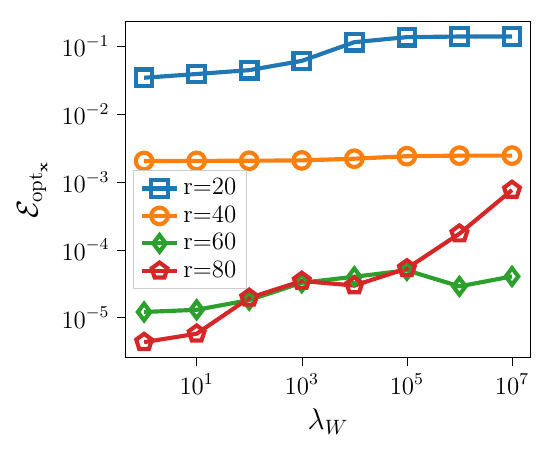}\hspace{0.3cm}
    \includegraphics[width=0.45\linewidth]{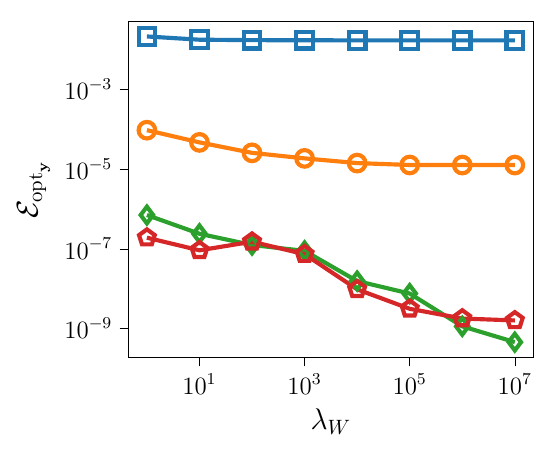}
    \caption{Toda Lattice Model: Optimization errors by \text{pH-OpInf-W} of $r$ dimensions: $\mathcal{E}_{\text{opt}_\bx}$ (left) and $\mathcal{E}_{\text{opt}_{\by}}$ (right).}
    \label{fig:nonlinear_W_lambba_test}
\end{figure}

Next, we examine the effect of the regularization parameter $\lambda_R$ in the optimization problem 
\eqref{eq:opt1}--\eqref{eq:opt2}, \text{pH-OpInf-R}. To this end, we pick eight values of $\lambda_R$ from the set $\{10^{-14}, 10^{-13}, \dots, 10^{-7}\}$ for each value of $r$ from $\{20,40,60,80\}$. Figure \ref{fig:nonlinear_R_lambba_test} shows the corresponding optimization errors $\mathcal{E}_{\text{opt}_\bx}$ and $\mathcal{E}_{\text{opt}_{\by}}$. For $r=20$ and 40, the errors remain relatively steady as $\lambda_R$ varies. However, for the larger dimensions ($r=60$ and 80), both optimization errors initially decrease and then increase after $\lambda_R=10^{-11}$, consistent with the behavior observed in the linear case. Based on these observations, we set $\lambda_R=10^{-11}$ in the subsequent experiments. 
\begin{figure}[!ht]
    \centering
    \includegraphics[width=0.45\linewidth]{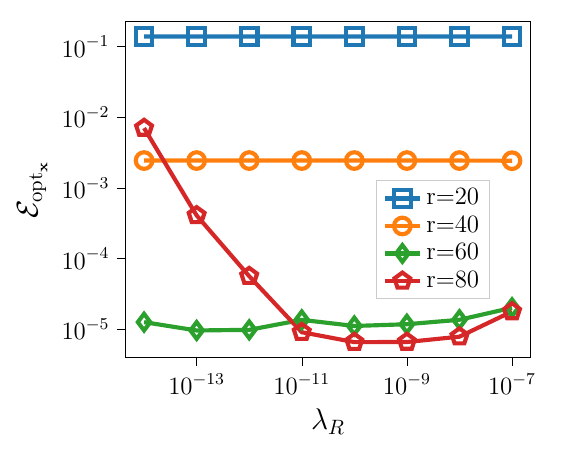}\hspace{0.3cm}
    \includegraphics[width=0.45\linewidth]{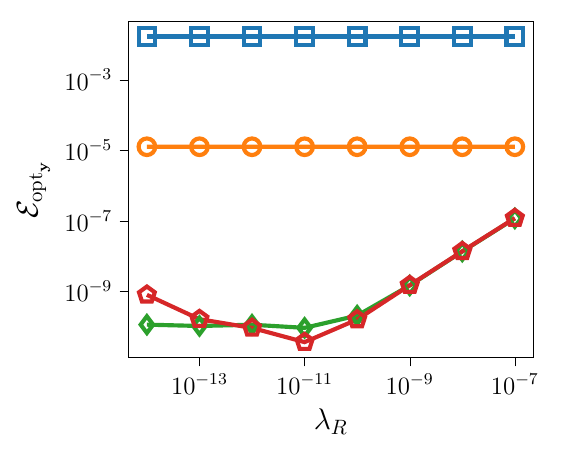}
    \caption{Toda Lattice Model: Optimization errors by \text{pH-OpInf-R} of $r$ dimensions: $\mathcal{E}_{\text{opt}_\bx}$ (left) and $\mathcal{E}_{\text{opt}_{\by}}$ (right).}
    \label{fig:nonlinear_R_lambba_test}
\end{figure}

With $\lambda_W=10^3$ in \text{pH-OpInf-W} and $\lambda_R=10^{-11}$ in \text{pH-OpInf-R}, the minimum eigenvalues of the obtained $\bR_r$ are list in Table \ref{tab:nonlinear_min_eigvalue} for $r=20,40,60$ and 80, confirming the positive semi-definiteness of the inferred $\bR_r$ in both approaches.

\begin{table}[!ht]
\centering
\caption{Toda Lattice Model: Minimum eigenvalues of $\bR_r$ obtained by \text{pH-OpInf-W} and \text{pH-OpInf-R} for different values of $r$.}
\begin{tabular}{c|c|c|c|c}
 & r=20        & r=40       & r=60       & r=80       \\ \hline
\text{pH-OpInf-W}       & 1.962 $\times 10^{-7}$ & 3.857$\times 10^{-6}$& 4.103$\times 10^{-4}$ & 4.074$\times 10^{-4}$ \\
\text{pH-OpInf-R}        & 2.250$\times 10^{-10}$ & 2.924$\times 10^{-6}$ & 1.637$\times 10^{-4}$ & 1.168$\times 10^{-4}$ 
\end{tabular}
\label{tab:nonlinear_min_eigvalue}
\end{table}

\subsubsection*{Test 2. Illustration of the error estimates}\label{sec:error_estimate}
For this nonlinear pH system, we first consider \text{pH-OpInf-ROM} without using the hyper-reduction. To illustrate the error estimation given in Corollary \ref{corollary1}, we choose a small time step size $\Delta t = 0.0025$ in both FOM and ROM simulations, set $T_{\text{FOM}}=T_{\text{ROM}}=50$ and vary the dimension $r$ of the \text{pH-OpInf-ROM} from 10 to 100, increasing incrementally by 10. 
In this case, the singular values of the snapshot matrix $\bX$ decay more slowly than those in the mass-spring-damper case: the first 10 POD modes capture 97.97\% of snapshot energy, the first 20 capture 99.8\%, and the first 30 already capture 99.99\%. 
The associated numerical errors defined in \eqref{eq:e_x_approx}--\eqref{eq:e_opt_y} are plotted in Figure \ref{fig:nonlinear_error_estimate}: 
the left column corresponds to \text{pH-OpInf-W}, and the right column corresponds to \text{pH-OpInf-R}. 
Note that the magnitude of the state vectors is 2.225, while that of the output is 0.483. 
The top row displays $\mathcal{E}_\bx$, $\mathcal{E}_{\text{proj}_\bx}$ and $\mathcal{E}_{\text{opt}_\bx}$. For both \text{pH-OpInf-W} and \text{pH-OpInf-R}, $\mathcal{E}_{\text{opt}_\bx}$ decreases as the dimension $r$ increases up to $r=60$, beyond which it saturates. However, $\mathcal{E}_{\text{proj}_\bx}$ continues to decline until $r=90$, for which it reaches machine precision. As a result, the state approximation error $\mathcal{E}_\bx$ is dominated by $\mathcal{E}_{\text{opt}_\bx}$ and follows a similar trend.
The output approximation error $\mathcal{E}_\by$ is shown alongside $\mathcal{E}_{\text{opt}_\by}$ and $\mathcal{E}_\bx$ in the middle row, from which we observe that  the optimization error $\mathcal{E}_{\text{opt}_{\by}}$ from \text{pH-OpInf-R} is smaller than that from \text{pH-OpInf-W} when $r$ is bigger than $40$. In both cases, $\mathcal{E}_{\text{opt}_{\by}}$ is smaller than $\mathcal{E}_\bx$, thus the output approximation error $\mathcal{E}_\by$ is primarily influenced by $\mathcal{E}_\bx$ and follows a similar pattern. 
The bottom row in Figure \ref{fig:nonlinear_error_estimate} compares the approximation errors of \text{pH-OpInf-ROM} with the intrusive \text{SP-G-ROM}.
The results show that the \text{pH-OpInf-ROM}, based on either \text{pH-OpInf-W} or \text{pH-OpInf-R}, achieves better approximations than the instrusive \text{SP-G-ROM} for dimensions up to a fairly large dimension $r = 80$.

\begin{figure}[!ht]
\centering
  \includegraphics[width=0.45\linewidth]{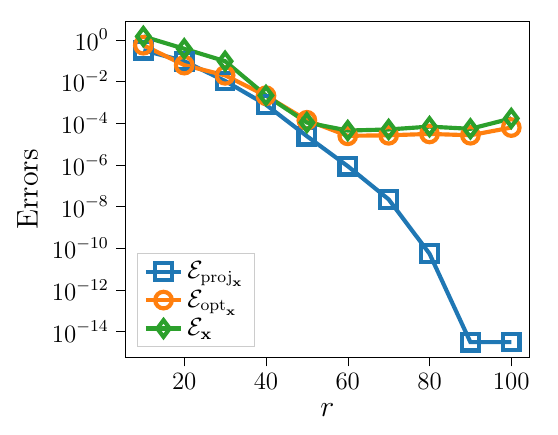}\hspace{0.3cm} 
   \includegraphics[width=0.45\linewidth]{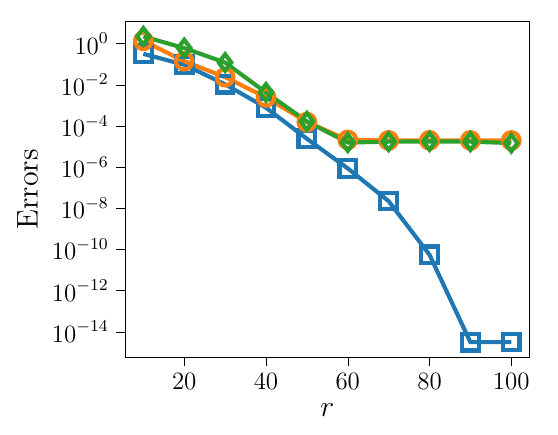} 
  \includegraphics[width=0.45\linewidth]{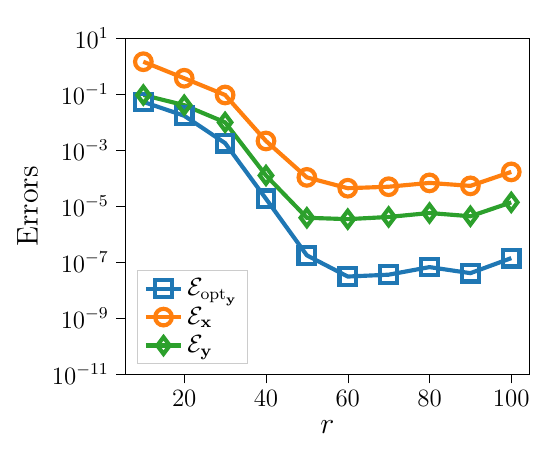} \hspace{0.3cm} 
  \includegraphics[width=0.45\linewidth]{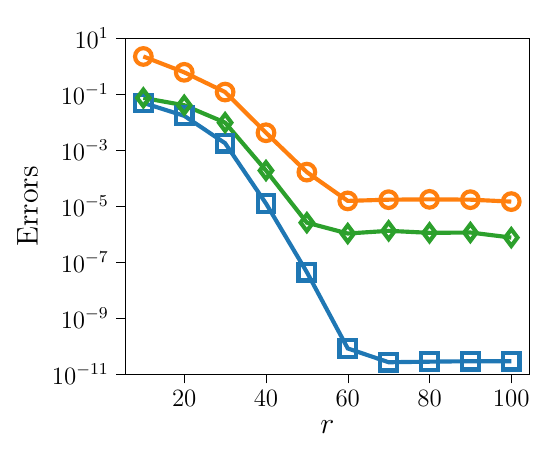}  
  \includegraphics[width=0.45\linewidth]{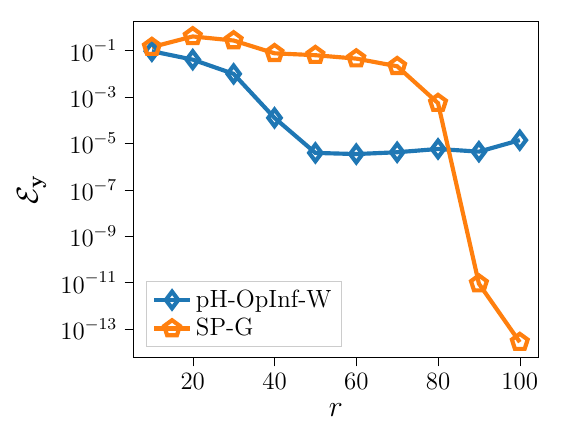} \hspace{0.3cm}
  \includegraphics[width=0.45\linewidth]{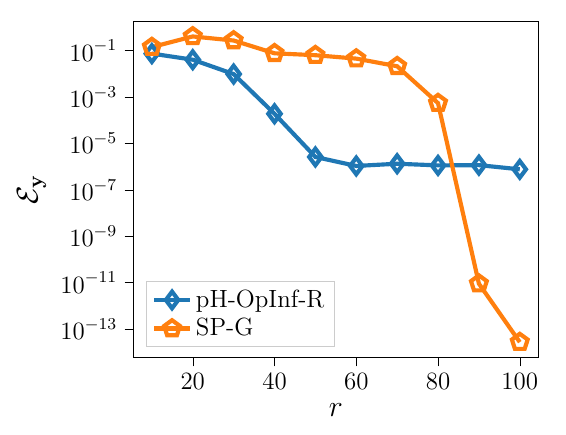} 
\caption{Toda Lattice Model: Numerical errors of \text{pH-OpInf-W} (left) and \text{pH-OpInf-R} (right) of $r$ dimensions when $T_{\text{FOM}}=T_{\text{ROM}}=50$.}
\label{fig:nonlinear_error_estimate}
\end{figure}

The time evolution of the output and approximate Hamiltonian from the \text{pH-OpInf-ROMs} is presented in Figure \ref{fig:nonlinear_ROM_energy_Tf50} when $r= 20$, $40$ and $60$, respectively. 
We observe that the reduced-order approximations, using either \text{pH-OpInf-W} or \text{pH-OpInf-R}, get more accurate as $r$ increases, and when $r= 40$, they closely align with those from the FOM simulation. Similar to the linear case in Section \ref{sec:linear case}, although both achieve good approximations, the \text{pH-OpInf-R} attains a smaller optimization error $\mathcal{E}_{\text{opt}_{\by}}$. Therefore, we use  \text{pH-OpInf-R} in the subsequent experiments.
\begin{figure}[!ht]
\centering
    \includegraphics[width=0.45\linewidth]{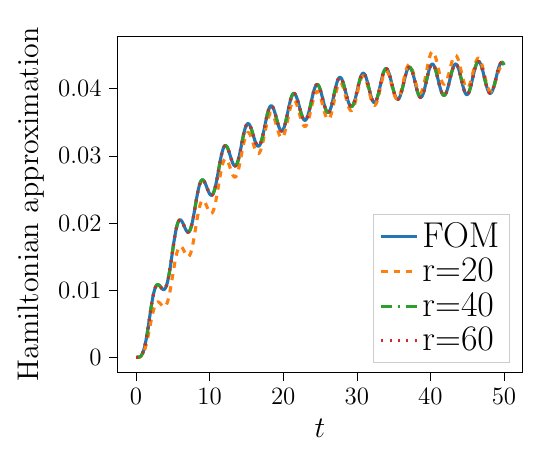}\hspace{0.3cm}
    \includegraphics[width=0.45\linewidth]{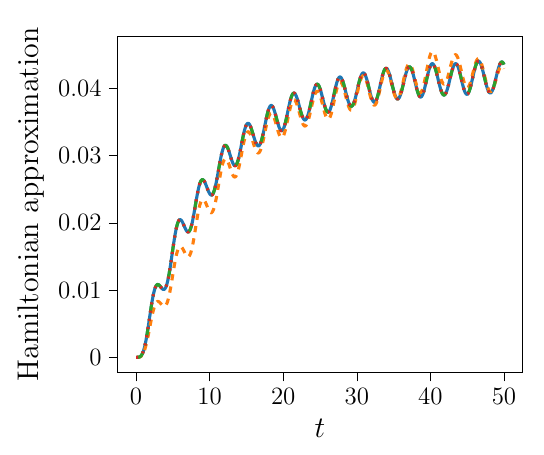}
    \includegraphics[width=0.45\linewidth]{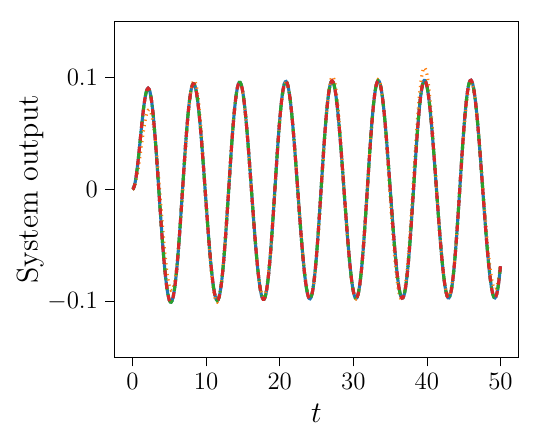}\hspace{0.3cm}
    \includegraphics[width=0.45\linewidth]{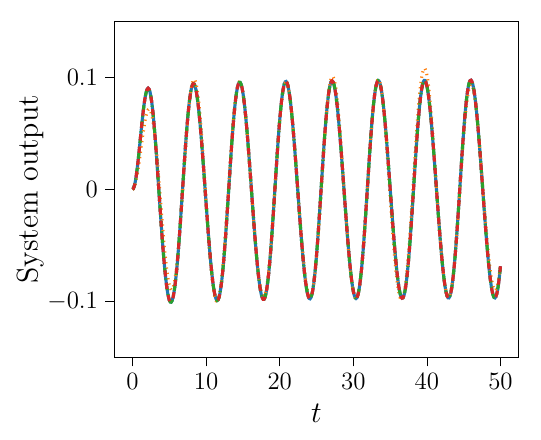}
    \caption{Toda Lattice Model: Hamiltonian approximation (top row) and system output (bottom row) from \text{pH-OpInf-W} (left) and \text{pH-OpInf-R} (right) of $r$ dimensions, along with those of the FOM. All subplots share the same legend shown in the first one.}
    \label{fig:nonlinear_ROM_energy_Tf50}
\end{figure}

\subsubsection*{Test 3. Performance of hyper-reduction}
Next, we apply hyper-reduction and test the performance of \text{pH-OpInf-DEIM}. The nonlinear Hamiltonian function in \eqref{eq:Toda energy} is recast to
\begin{equation*}
    H(\bx) =  \frac{1}{2}\bx^\intercal \bQ \bx + {\bc}^\intercal h(\bx), \quad \text{where} \quad
\mathbf{Q}= 
\left[ \begin{matrix}
\mathbf{0} & \mathbf{0}  \\
\mathbf{0}& \mathbf{I}_{N_0} 
\end{matrix} \right]_{n\times n}
\end{equation*}
and 
$\bc = [1,1,\dots,1]^\intercal\in \mathbb{R}^{N_0}$ with $h_1 = \exp(q_1-q_2) -q_1-N_0$, and $h_i = \exp(q_i-q_{i+1})$ for $i=2,3,\dots,N_0-1$ and $h_{N_0} = \exp(q_{N_0})$.
Choosing $r=20$ and $60$, respectively, we simulate the \text{pH-OpInf-DEIM} using $m$ interpolation points in DEIM and compare it with the  \text{pH-OpInf-ROM} of the same dimensions. The numerical errors in state and output approximations are listed in Table \ref{tab:DEIM_errors}, together with the DEIM interpolation error $\mathcal{E}_{\text{DEIM}}$. Overall, $\mathcal{E}_{\text{DEIM}}$ decreases quickly as $m$ increases. For both $r=20$ and $60$, its precision reaches $10^{-11}$ or higher when $m = 60$, and the \text{pH-OpInf-DEIM} results coincide with those of \text{pH-OpInf-ROM}. This observation matches our theoretical results, as stated in Theorems \ref{thm:1} and \ref{thm:2}, because, comparing with \text{pH-OpInf-ROM}, the error bound of \text{pH-OpInf-DEIM} involves the additional term $\mathcal{E}_{\text{DEIM}}$. When $\mathcal{E}_{\text{DEIM}}$ is sufficiently small, errors from both ROMs should become identical. 
\begin{table}[!ht]
    \centering
    \begin{subtable}{\textwidth}
        \centering
        \begin{tabular}{c|c|c|c|c|c}
                             & \multirow{2}{3cm}{\centering \text{pH-OpInf-ROM} with $r=20$ } & \multicolumn{4}{c}{\text{pH-OpInf-DEIM} with $r=20$}          
                             \\ \cline{3-6} 
                             &      & $m=30$     & $m=40$     & $m=50$      & $m=60$  \\ \hline
        $\mathcal{E}_\mathbf{x}$      & 6.042 $\times 10^{-1}$ &6.494$\times 10^{-1}$   &5.958$\times 10^{-1}$ & 6.042$\times 10^{-1}$ &6.042$\times 10^{-1}$\\ \hline
        $\mathcal{E}_\mathbf{y}$      & 4.032$\times 10^{-2}$ & 4.400$\times 10^{-2}$   & 4.030$\times 10^{-2}$ & 4.032$\times 10^{-2}$  & 4.032$\times 10^{-2}$\\ \hline
        $\mathcal{E}_{\text{DEIM}}$ &    -      & 3.985$\times 10^{-2}$ &   7.552$\times 10^{-4}$&4.095$\times 10^{-8}$ & 4.585$\times 10^{-15}$ 
        \end{tabular}
    \end{subtable}%
     \vspace{0.2cm}   
    \begin{subtable}{\textwidth}
        \centering
        \begin{tabular}{c|c|c|c|c|c}
                             & \multirow{2}{3cm}{\centering \text{pH-OpInf-ROM} with $r=60$ } & \multicolumn{4}{c}{\text{pH-OpInf-DEIM} with $r=60$}           
                             \\ \cline{3-6}
                             &      & $m=50$     & $m=55$      & $m=60$   & $m=65$   \\ \hline
        $\mathcal{E}_\mathbf{x}$      & 1.719$\times 10^{-4}$  & 1.745$\times 10^{-4}$  &   1.719$\times 10^{-4}$   &  1.719$\times 10^{-4}$ &   1.719$\times 10^{-4}$ \\ \hline
        $\mathcal{E}_\mathbf{y}$      & 7.748$\times 10^{-6}$  & 7.781$\times 10^{-6}$  & 7.748$\times 10^{-6}$   &7.748$\times 10^{-6}$ &7.748$\times 10^{-6}$    \\ \hline
        $\mathcal{E}_{\text{DEIM}}$ &    -      & 1.184$\times 10^{-5}$&   4.384$\times 10^{-9}$ & 3.478$\times 10^{-11}$ & 1.737$\times 10^{-13}$ 
        \end{tabular}
     \end{subtable}   
    \caption{Toda Lattice Model: Errors of \text{pH-OpInf-ROM} with dimension $r$ and \text{pH-OpInf-DEIM} with the same dimension using $m$ DEIM interpolation points.}
    \label{tab:DEIM_errors}
\end{table}


\subsubsection*{Test 4. Performance at a different input} \label{sec:train_test}
We test \text{pH-OpInf-ROM} and \text{pH-OpInf-DEIM} using an input that differs from the training input. Particularly, to infer the reduced operators, we use the training input $u(t) = 0.1 \sin(t)$, as shown in Figure \ref{fig:nonlinear_train_test_input} (left), and simulate the reduced-order systems with a testing input, as shown in Figure \ref{fig:nonlinear_train_test_input} (right). The sawtooth signal is discontinuous, which makes the test challenging.
Figure \ref{fig:nonlinear_train_test_sol_energy} compares the system output and the Hamiltonian function obtained by simulating the \text{pH-OpInf-ROM} with dimension $r=50$ and the \text{pH-OpInf-DEIM} with the same dimension and $m=50$, along with the FOM results. Both \text{pH-OpInf-ROM} and \text{pH-OpInf-DEIM} are able to accurately capture the system output in response to the testing input and effectively approximate the Hamiltonian function during the simulation.

\begin{figure}[!ht]
\centering
  \includegraphics[width=0.45\linewidth]{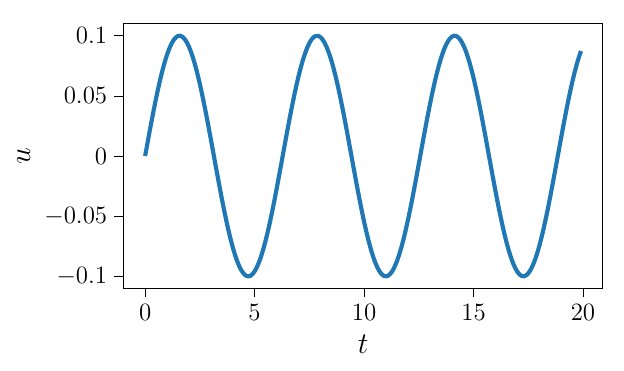} 
\hspace{0.3cm}
  \includegraphics[width=0.45\linewidth]{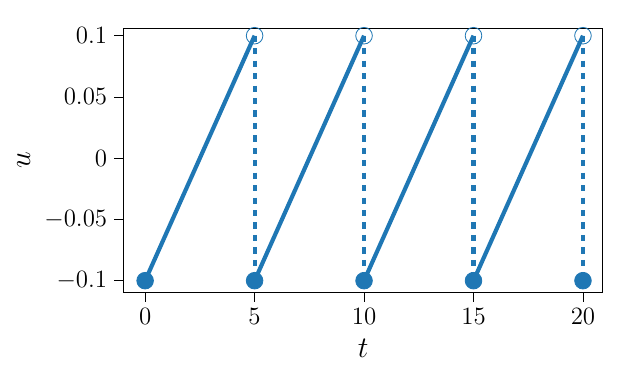}
\caption{Toda Lattice Model: Training input (left) and testing input (right) for Test 4.}
\label{fig:nonlinear_train_test_input}
\end{figure}

\begin{figure}[!ht]
\centering
  \includegraphics[width=0.45\linewidth]{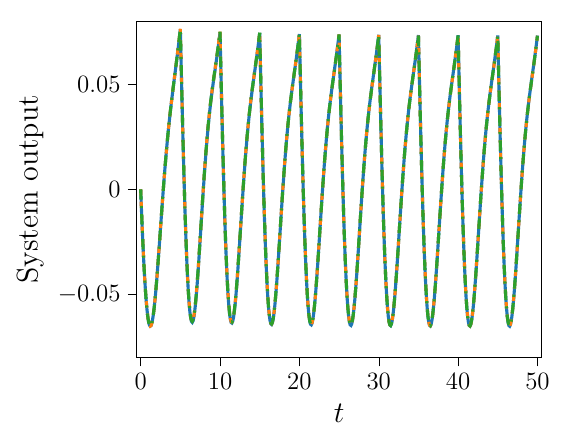}\hspace{0.3cm}  
  \includegraphics[width=0.42\linewidth]{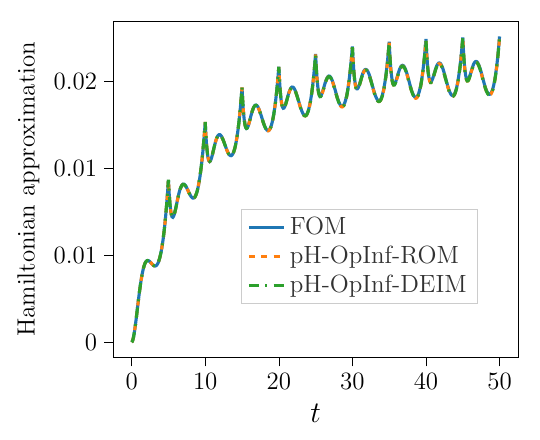}
\caption{Toda Lattice Model: Simulation results of \text{pH-OpInf-ROM} for $r=50$, \text{pH-OpInf-DEIM} for $r=50$ and $m=50$, and the FOM: system output (left) and Hamiltonian approximation (right) for Test 4.}
\label{fig:nonlinear_train_test_sol_energy}
\end{figure}

\section{Conclusions} \label{sec: conclusion}
In this work, we have extended the operator inference framework to learn structure-preserving reduced-order models for port-Hamiltonian systems. By leveraging data of the system's state, input, and output, as well as the Hamiltonian function, we formulate two optimization problems and corresponding ROMs,  \text{pH-OpInf-W} and \text{pH-OpInf-R}, for inferring reduced-order operators that retain the physical properties and geometric structure of the original system. 
The former finds reduced operators, $\bB_r$ and $\bD_r$ (hence $\bJ_r$ and $\bR_r$), simultaneously by solving a semi-definite optimization, while the latter decouples the optimization into two steps: the first step determines the reduced operator $\bB_r$ using regularized least squares, and the second step identifies the other reduced operator $\bD_r$ through a semi-definite optimization. 
Based on the reduced-order operators, the low-dimensional \text{pH-OpInf-ROM} can be constructed. To further address the challenges of evaluating nonlinear terms in the ROM, we use a DEIM-based hyper-reduction method, which leads to the \text{pH-OpInf-DEIM} model. We analyze the corresponding approximation errors in the system's state and output and and numerically verify them through experiments on a linear mass-spring-damper system and a nonlinear Toda lattice model.

The numerical results show that both \text{pH-OpInf-W} and \text{pH-OpInf-R} with carefully chosen optimization parameters find reduced-order operators that preserve the appropriate structures.  
Generally, the overall performance of these two approaches is comparable, since the \text{pH-OpInf-ROM} constructed using either approach yields accurate outputs in response to inputs. However, \text{pH-OpInf-R} achieves a smaller optimization error for the output, $\mathcal{E}_{\text{opt}_\by}$, than \text{pH-OpInf-W}, which makes it the preferred choice.   
Furthermore, for nonlinear pH systems, the \text{pH-OpInf-DEIM} model using sufficient DEIM interpolation points is able to achieve the same performance at a reduced computational cost. 
We note that alternative hyper-reduction approaches, such as ECSW \cite{farhat2015structure}, may also be applied to inferred ROMs when they are nonlinear. We will investigate their performance in future work.

\section*{Acknowledgement}
Y.G. was partially supported by the SPARC grant from the University of South Carolina under award number CL071-130600-N1400-202-80006259-1.
L.J. was partially supported by the U.S. Department of Energy under award numbers DE-SC0022254 and DE-SC0025527.
B.K. and Z.W. were  supported by U.S. Office of Naval Research under award number N00014-22-1-2624. 
Z.W. was partially supported by U.S. National Science Foundation under award numbers
DMS-2012469, DMS-2038080, DMS-2245097 and an ASPIRE grant from the Office of the Vice President for Research at the University of South Carolina.


\bibliographystyle{elsarticle-num}
\bibliography{ref.bib}

\begin{thebibliography}{10}
\expandafter\ifx\csname url\endcsname\relax
  \def\url#1{\texttt{#1}}\fi
\expandafter\ifx\csname urlprefix\endcsname\relax\def\urlprefix{URL }\fi
\expandafter\ifx\csname href\endcsname\relax
  \def\href#1#2{#2} \def\path#1{#1}\fi

\bibitem{van2014port}
A.~Van Der~Schaft, D.~Jeltsema, et~al., Port-{H}amiltonian systems theory: An
  introductory overview, Foundations and Trends{\textregistered} in Systems and
  Control 1~(2-3) (2014) 173--378.

\bibitem{duindam2009modeling}
V.~Duindam, A.~Macchelli, S.~Stramigioli, H.~Bruyninckx, Modeling and control
  of complex physical systems: the port-{H}amiltonian approach, Springer
  Science \& Business Media, 2009.

\bibitem{hairer2006geometric}
E.~Hairer, M.~Hochbruck, A.~Iserles, C.~Lubich, Geometric numerical
  integration, Oberwolfach Reports 3~(1) (2006) 805--882.

\bibitem{quarteroni2015reduced}
A.~Quarteroni, A.~Manzoni, F.~Negri, Reduced basis methods for partial
  differential equations: an introduction, Vol.~92, Springer, 2015.

\bibitem{holmes2012turbulence}
P.~Holmes, Turbulence, coherent structures, dynamical systems and symmetry,
  Cambridge university press, 2012.

\bibitem{kutz2016dynamic}
J.~N. Kutz, S.~L. Brunton, B.~W. Brunton, J.~L. Proctor, Dynamic mode
  decomposition: data-driven modeling of complex systems, SIAM, 2016.

\bibitem{schmid2022dynamic}
P.~J. Schmid, Dynamic mode decomposition and its variants, Annual Review of
  Fluid Mechanics 54~(1) (2022) 225--254.

\bibitem{kramer2024learning}
B.~Kramer, B.~Peherstorfer, K.~E. Willcox, Learning nonlinear reduced models
  from data with operator inference, Annual Review of Fluid Mechanics 56 (2024)
  521--548.

\bibitem{peherstorfer2016data}
B.~Peherstorfer, K.~Willcox, Data-driven operator inference for nonintrusive
  projection-based model reduction, Computer Methods in Applied Mechanics and
  Engineering 306 (2016) 196--215.

\bibitem{antoulas2005approximation}
A.~C. Antoulas, Approximation of large-scale dynamical systems, SIAM, 2005.

\bibitem{antoulas2020interpolatory}
A.~C. Antoulas, C.~A. Beattie, S.~G{\"u}{\u{g}}ercin, Interpolatory methods for
  model reduction, SIAM, 2020.

\bibitem{benner2015survey}
P.~Benner, S.~Gugercin, K.~Willcox, A survey of projection-based model
  reduction methods for parametric dynamical systems, SIAM review 57~(4) (2015)
  483--531.

\bibitem{hesthaven2016certified}
J.~S. Hesthaven, G.~Rozza, B.~Stamm, et~al., Certified reduced basis methods
  for parametrized partial differential equations, Vol. 590, Springer, 2016.

\bibitem{lall2003structure}
S.~Lall, P.~Krysl, J.~E. Marsden, Structure-preserving model reduction for
  mechanical systems, Physica D: Nonlinear Phenomena 184~(1-4) (2003) 304--318.

\bibitem{carlberg2015preserving}
K.~Carlberg, R.~Tuminaro, P.~Boggs, Preserving {L}agrangian structure in
  nonlinear model reduction with application to structural dynamics, SIAM
  Journal on Scientific Computing 37~(2) (2015) B153--B184.

\bibitem{gong2017structure}
Y.~Gong, Q.~Wang, Z.~Wang, Structure-preserving {G}alerkin {POD} reduced-order
  modeling of {H}amiltonian systems, Computer Methods in Applied Mechanics and
  Engineering 315 (2017) 780--798.

\bibitem{miyatake2019structure}
Y.~Miyatake, Structure-preserving model reduction for dynamical systems with a
  first integral, Japan Journal of Industrial and Applied Mathematics 36~(3)
  (2019) 1021--1037.

\bibitem{barbulescu2021efficient}
R.~Barbulescu, G.~Ciuprina, T.~Ionescu, D.~Ioan, L.~M. Silveira, Efficient
  model reduction of myelinated compartments as port-{H}amiltonian systems, in:
  Scientific Computing in Electrical Engineering: SCEE 2020, Eindhoven, The
  Netherlands, February 2020, Springer, 2021, pp. 3--12.

\bibitem{beattie2011structure}
C.~Beattie, S.~Gugercin, Structure-preserving model reduction for nonlinear
  port-{H}amiltonian systems, in: 2011 50th IEEE conference on decision and
  control and European control conference, IEEE, 2011, pp. 6564--6569.

\bibitem{chaturantabut2016structure}
S.~Chaturantabut, C.~Beattie, S.~Gugercin, Structure-preserving model reduction
  for nonlinear port-{H}amiltonian systems, SIAM Journal on Scientific
  Computing 38~(5) (2016) B837--B865.

\bibitem{peng2016symplectic}
L.~Peng, K.~Mohseni, Symplectic model reduction of {H}amiltonian systems, SIAM
  Journal on Scientific Computing 38~(1) (2016) A1--A27.

\bibitem{gruber2025variationally}
A.~Gruber, I.~Tezaur, Variationally consistent {H}amiltonian model reduction,
  SIAM Journal on Applied Dynamical Systems 24~(1) (2025) 376--414.

\bibitem{afkham2017structure}
B.~M. Afkham, J.~S. Hesthaven, Structure preserving model reduction of
  parametric {H}amiltonian systems, SIAM Journal on Scientific Computing 39~(6)
  (2017) A2616--A2644.

\bibitem{hesthaven2021structure}
J.~Hesthaven, C.~Pagliantini, Structure-preserving reduced basis methods for
  {P}oisson systems, Mathematics of Computation 90~(330) (2021) 1701--1740.

\bibitem{pagliantini2021dynamical}
C.~Pagliantini, Dynamical reduced basis methods for {H}amiltonian systems,
  Numerische Mathematik 148~(2) (2021) 409--448.

\bibitem{hesthaven2022reduced}
J.~S. Hesthaven, C.~Pagliantini, G.~Rozza, Reduced basis methods for
  time-dependent problems, Acta Numerica 31 (2022) 265--345.

\bibitem{barrault2004empirical}
M.~Barrault, Y.~Maday, N.~C. Nguyen, A.~T. Patera, An ‘empirical
  interpolation’method: application to efficient reduced-basis discretization
  of partial differential equations, Comptes Rendus Mathematique 339~(9) (2004)
  667--672.

\bibitem{chaturantabut2010nonlinear}
S.~Chaturantabut, D.~C. Sorensen, Nonlinear model reduction via discrete
  empirical interpolation, SIAM Journal on Scientific Computing 32~(5) (2010)
  2737--2764.

\bibitem{pagliantini2023gradient}
C.~Pagliantini, F.~Vismara, Gradient-preserving hyper-reduction of nonlinear
  dynamical systems via discrete empirical interpolation, SIAM Journal on
  Scientific Computing 45~(5) (2023) A2725--A2754.

\bibitem{sharma2023symplectic}
H.~Sharma, H.~Mu, P.~Buchfink, R.~Geelen, S.~Glas, B.~Kramer, Symplectic model
  reduction of {H}amiltonian systems using data-driven quadratic manifolds,
  Computer Methods in Applied Mechanics and Engineering 417 (2023) 116402.

\bibitem{yildiz2024data}
S.~Y{\i}ld{\i}z, P.~Goyal, T.~Bendokat, P.~Benner, Data-driven identification
  of quadratic representations for nonlinear {H}amiltonian systems using weakly
  symplectic liftings, Journal of Machine Learning for Modeling and Computing
  5~(2).

\bibitem{buchfink2023symplectic}
P.~Buchfink, S.~Glas, B.~Haasdonk, Symplectic model reduction of {H}amiltonian
  systems on nonlinear manifolds and approximation with weakly symplectic
  autoencoder, SIAM Journal on Scientific Computing 45~(2) (2023) A289--A311.

\bibitem{sharma2022hamiltonian}
H.~Sharma, Z.~Wang, B.~Kramer, {H}amiltonian operator inference:
  Physics-preserving learning of reduced-order models for canonical
  {H}amiltonian systems, Physica D: Nonlinear Phenomena 431 (2022) 133122.

\bibitem{gruber2023canonical}
A.~Gruber, I.~Tezaur, Canonical and noncanonical {H}amiltonian operator
  inference, Computer Methods in Applied Mechanics and Engineering 416 (2023)
  116334.

\bibitem{geng2024gradient}
Y.~Geng, J.~Singh, L.~Ju, B.~Kramer, Z.~Wang, Gradient preserving operator
  inference: Data-driven reduced-order models for equations with gradient
  structure, Computer Methods in Applied Mechanics and Engineering 427 (2024)
  117033.

\bibitem{vijaywargiya2025tensor}
A.~Vijaywargiya, S.~A. McQuarrie, A.~Gruber, Tensor parametric {H}amiltonian
  operator inference, arXiv preprint arXiv:2502.10888.

\bibitem{cherifi2022non}
K.~Cherifi, P.~K. Goyal, P.~Benner, A non-intrusive method to inferring linear
  port-{H}amiltonian realizations using time-domain data, Electronic
  Transactions on Numerical Analysis: Special Issue SciML 56 (2022) 102--116.

\bibitem{morandin2023port}
R.~Morandin, J.~Nicodemus, B.~Unger, Port-{H}amiltonian dynamic mode
  decomposition, SIAM Journal on Scientific Computing 45~(4) (2023)
  A1690--A1710.

\bibitem{berkooz1993proper}
G.~Berkooz, P.~Holmes, J.~L. Lumley, The proper orthogonal decomposition in the
  analysis of turbulent flows, Annual review of fluid mechanics 25~(1) (1993)
  539--575.

\bibitem{gugercin2012structure}
S.~Gugercin, R.~V. Polyuga, C.~Beattie, A.~Van Der~Schaft, Structure-preserving
  tangential interpolation for model reduction of port-{H}amiltonian systems,
  Automatica 48~(9) (2012) 1963--1974.

\bibitem{drmac2016new}
Z.~Drmac, S.~Gugercin, A new selection operator for the discrete empirical
  interpolation method---improved a priori error bound and extensions, SIAM
  Journal on Scientific Computing 38~(2) (2016) A631--A648.

\bibitem{helmberg1996interior}
C.~Helmberg, F.~Rendl, R.~J. Vanderbei, H.~Wolkowicz, An interior-point method
  for semidefinite programming, SIAM Journal on optimization 6~(2) (1996)
  342--361.

\bibitem{chaturantabut2012state}
S.~Chaturantabut, D.~C. Sorensen, A state space error estimate for {POD-DEIM}
  nonlinear model reduction, SIAM Journal on numerical analysis 50~(1) (2012)
  46--63.

\bibitem{dahlquist1958stability}
G.~Dahlquist, Stability and error bounds in the numerical integration of
  ordinary differential equations, Ph.D. thesis, Almqvist \& Wiksell (1958).

\bibitem{kunisch2001galerkin}
K.~Kunisch, S.~Volkwein, Galerkin proper orthogonal decomposition methods for
  parabolic problems, Numerische mathematik 90 (2001) 117--148.

\bibitem{singler2014new}
J.~R. Singler, New {POD} error expressions, error bounds, and asymptotic
  results for reduced order models of parabolic {PDE}s, SIAM Journal on
  Numerical Analysis 52~(2) (2014) 852--876.

\bibitem{peherstorfer2020sampling}
B.~Peherstorfer, Sampling low-dimensional {M}arkovian dynamics for
  preasymptotically recovering reduced models from data with operator
  inference, SIAM Journal on Scientific Computing 42~(5) (2020) A3489--A3515.

\bibitem{polyuga2010model}
R.~V. Polyuga, Model reduction of port-{H}amiltonian systems.

\bibitem{aps2019mosek}
M.~ApS, Mosek optimization toolbox for matlab, User’s Guide and Reference
  Manual, Version 4~(1).

\bibitem{diamond2016cvxpy}
S.~Diamond, S.~Boyd, Cvxpy: A python-embedded modeling language for convex
  optimization, Journal of Machine Learning Research 17~(83) (2016) 1--5.

\bibitem{agrawal2018rewriting}
A.~Agrawal, R.~Verschueren, S.~Diamond, S.~Boyd, A rewriting system for convex
  optimization problems, Journal of Control and Decision 5~(1) (2018) 42--60.

\bibitem{farhat2015structure}
C.~Farhat, T.~Chapman, P.~Avery, Structure-preserving, stability, and accuracy
  properties of the energy-conserving sampling and weighting method for the
  hyper reduction of nonlinear finite element dynamic models, International
  journal for numerical methods in engineering 102~(5) (2015) 1077--1110.

\end{thebibliography}

\end{document}